\documentclass[11pt]{article}

\usepackage{geometry}
\usepackage{amsmath} 
\usepackage{amsfonts}
\usepackage{amsthm}
\usepackage{amssymb} 
\usepackage{graphicx}
\usepackage{graphicx}

\newtheorem{theorem}{Theorem}
\newtheorem{remark}{Remark}
\newtheorem{lemma}{Lemma}

\newtheorem{corollary}{Corollary}

\def\D{\displaystyle} 

\newcommand{\hx}{{\mathbf{e}}_1}
\newcommand{\hy}{{\mathbf{e}}_2}
\newcommand{\hz}{{\mathbf{e}}_3}

\newcommand{\bfO}{\boldsymbol{0}}
\newcommand{\bfx}{\mathbf{x}}
\newcommand{\bfy}{\mathbf{y}}
\newcommand{\bfz}{\mathbf{z}}
\newcommand{\bfp}{\mathbf{p}}
\newcommand{\bfq}{\mathbf{q}}
\newcommand{\bfu}{\mathbf{u}}
\newcommand{\bfv}{\mathbf{v}}
\newcommand{\bfw}{\mathbf{w}}

\newcommand{\bff}{\mathbf{f}}
\newcommand{\bfg}{\mathbf{g}}
\newcommand{\bfh}{\mathbf{h}}

\newcommand{\bfE}{\mathbf{E}}

\newcommand{\bfU}{\mathbf{U}}
\newcommand{\bfV}{\mathbf{V}}
\newcommand{\bfQ}{\mathbf{Q}}
\newcommand{\bfH}{\mathbf{H}}

\newcommand{\bfM}{\mathbf{M}}
\newcommand{\bfN}{\mathbf{N}}

\newcommand{\bbG}{\mathbb{G}}

\newcommand{\bbI}{\mathbb{I}}
\newcommand{\bbJ}{\mathbb{J}}

\newcommand{\bbR}{\mathbb{R}}

\newcommand{\nc}{\boldsymbol{\nu}_0} 
\newcommand{\nd}{\boldsymbol{\nu}_D}        

\def\pd{\partial}                     
\def\V{\Vert}                         

\def\la{\langle}                      
\def\ra{\rangle}                      
\def\cl{\nabla\times}                 
\def\clx{\nabla_{\bfx}\times}         
\def\cly{\nabla_{\bfy}\times}         
\def\dv{\nabla\cdot}                  
\def\gd{\nabla}                       
\def\sgd{\nabla_{\Sigma}}             
\def\vgd{\vec{\nabla}_{\Sigma}\times} 
\def\dt{\Delta}                       

\newcommand{\tcl}{\mathbf{curl}}

\newcommand{\bfX}{\mathbf{X}}

\begin{document}

\title{Near field linear sampling method for an inverse problem in an electromagnetic waveguide}

\author{Peter Monk$^1$ \and Virginia Selgas$^2$ \and Fan Yang$^2$}
\date{%
   \small{$^1$Dept. of Mathematical Sciences, University of Delaware, Newark DE 19716, USA\\%
    $^2$Dept. de Matem\'aticas, Universidad de Oviedo, EPIG, c/Luis Ortiz Berrocal s/n,  33203 Gij\'on, Spain\\%
    $^3$Dept. of Mathematics, California Polytechnic State University, San Luis Obispo, CA 93407-0403, USA\\
 email: \texttt{monk@udel.edu}, \texttt{selgasvirginia@uniovi.es}, \texttt{fayang@calpoly.edu}\\[2ex]}%
    \today
}

\maketitle

\vspace{10pt}

\begin{abstract}
We consider the problem of determining the shape and location of an unknown penetrable object in a perfectly conducting electromagnetic waveguide.  The inverse problem is posed in the frequency domain and uses multistatic data in the near field. In particular, we assume that we are given measurements of the electric scattered field due to point sources on a cross-section of the waveguide and measured  on the same cross-section, which  is away from the scatterer but not in the far field.

The problem is solved by using the Linear Sampling Method (LSM)  and we also discuss the generalized LSM.  We start by
giving a brief discussion of the direct problem and its associated interior transmission problem. Then, we adapt and analyze the LSM to deal with the inverse problem. This {extends the work on the LSM for perfectly conducting scatterers  in a waveguide by one of us (Yang) to the detection of penetrable objects}.  We provide several useful results concerning reciprocity and
the density of fields due to single layer potentials.  We also prove the standard results for the LSM in the waveguide context. Finally we give numerical results to show the performance of the method for simple shapes.
\end{abstract}

%
%
%
%
%

\section{Introduction}
The detection or characterization of inaccessible objects in closed waveguides has received considerable attention in
recent years.  Much of the work has focused on the scalar acoustic problem. Examples include the use of time-reversal imaging~\cite{RouxFinck-utrasonicWG}, linearization methods~\cite{DediuMacLaughlin}, near-field measurements inside periodic waveguides~\cite{SunZheng-periodicWG},  and volume integral equations  with fixed point iteration~\cite{Yurys-IPguides}.  We  apply the Linear Sampling Method (LSM) in the frequency domain (see \cite{ColtonKress-IP} for background to this method).  In the context of waveguides, the use of the LSM was initiated by 
L. Bourgeois and E. Luneville~\cite{BLFD} who demonstrated the possibility of using the LSM to detect impenetrable obstacles. Two of us extended this work to penetrable obstacles and three dimensions~\cite{MonkSelgasFD}. {Then, one of us, F. Yang~\cite{FanPhD}, showed that the LSM can be applied to
reconstruct PEC scatterers in a closed waveguide.} It is the latter that is the main background for our current paper.  The LSM can also be used with time domain multistatic data, either by using the Fourier transform to move to the frequency domain~\cite{BourgTD} or by working in the time domain directly~\cite{MonkSelgasTD} but we shall not discuss the time domain here.  

It should be noted that far field data does not uniquely determine the scatterer  in an acoustic waveguide~\cite{ASBinvisible}.  {Using near field measurements as in this paper, we will prove that the inverse scattering problem has a unique solution}. Although we are working in the near field, the decay of evanescent modes implies that higher modes cannot be observed in the presence of noise. So non-uniqueness may be a practical problem even in the near field.  However, at least for the simple shapes examined in the  papers {discussed in the previous paragraph}, this does not seem to cause an issue (also multi-frequency or time domain data might ameliorate the problem).  It is reasonable to conjecture that uniqueness is also an issue for the electromagnetic inverse problem using far field measurements, but so far this has not been studied to our knowledge.

This paper is devoted to extending{~\cite{BLFD,MonkSelgasFD,FanPhD} to penetrable }electromagnetic scattering in the frequency domain using single-frequency multistatic data.  In particular, we study the model problem of determining the shape and the location of a bounded penetrable obstacle located in a tubular waveguide.  By a tubular waveguide we mean that the waveguide has a cross-section represented by a convex, open, bounded domain
$\Sigma\subseteq \mathbb{R}^{2}$ and the waveguide is the infinite domain $W=\Sigma\times\mathbb{R}$ having boundary $\Gamma$. The unknown scatterer occupies a bounded, open and Lipschitz domain $D\subset\Sigma\times\mathbb{R}$.  We assume that a probing 
electromagnetic field is due to point sources with arbitrary polarization  {that are} located on a surface 
\[\Sigma_r= \{ x\in\mathbb{R}^3\; ; \;(x_1,x_2)\in\Sigma , \, 
x_3=r  \} ,
\]
where $r$ is such that $\Sigma_r\cap \overline{D}=\emptyset$.  We define $\nc$ to be the normal to $\Sigma_r$ pointing in the direction of increasing $z$.
We also assume that measurements of the polarization, phase and amplitude of the resulting scattered field can be made on this surface.  From these multistatic data, we seek to determine the boundary of $D$ denoted $\partial D$.  We shall define the problem in more detail in the next section, but, to summarize the inverse problem we assume that the
scattered field $\bfu^s(\bfx;\bfy,\bfp)$ due to a point source at $\bfy\in\Sigma_r$ with polarization $\bfp$
and measured at $\bfx\in\Sigma_r$ is known (the data may also be corrupted with random noise, and in practice is only known for a finite number of source and receiver points on $\Sigma_r$.  From this data it is desired to determine $\partial D$.

The LSM  uses the near field operator defined for 
\[
\bfg\in L^2_T(\Sigma_r)=\{\bfg\in L^2(\Sigma_r)\;; \; \bfg\cdot\nc =0\mbox{ a.e. on }\Sigma_r\}
\]
by 
\begin{equation}\label{inv:nearfieldop} 
N\bfg(\bfx) \, = \, \int_{\Sigma_r}\nc(\bfx)\times\bfu^s(\bfx;\bfy,\bfg(\bfy))\, dS_{\bfy} \quad\mbox{for a.e. }\bfx\in\Sigma _r \, .
\end{equation}
We then consider the Near Field  Equation (NFE)
for $\bfg_z\in L^2_T(\Sigma_r)$ given by
\begin{equation} \label{inv:lsm:nfe}
    N\bfg_{\bfz}(\bfx) 
		= \nc\times\bfu^i(\bfx;\bfz,\bfq) 
\quad\mbox{a.e. } \bfx\in\Sigma _r\, ,
\end{equation}
where $\bfu^i(\bfx;\bfy,\bfp)$ is the field due to an auxiliary source point $\bfz\in W$ with polarization $\bfq$ in the {an empty}
waveguide.   Of course (\ref{inv:lsm:nfe}) is ill-posed but as usual for the LSM we shall show that there exists
an approximate solution to this equation such that $\bfz\mapsto\Vert \bfg_z\Vert_{L_T^2(\Sigma_r)}$ can be used as an indicator
function for $\partial D$ (as usual this is only a partial justification of the LSM, see \cite{ColtonKress-IP}). 

We intend
our inverse problem to be a simplified model  for applications of {inverse scattering in a waveguide   (see for example the engineering papers}~\cite{Dalarsson-PhD,SjobergLarsson}) although these applications are more complex.  Of course we are not the first to apply inverse scattering techniques to electromagnetic waveguides.  For example L. Borcea and D.-L. Nguyen~\cite{BorceaNguyen}
used reverse time migration and $\ell_1$ optimization to image objects in terminating waveguides.  This could be a
problem amenable to the application of the LSM but is not considered in our paper.  Also of interest J. Chen and G. Huang~\cite{ChenHuang} have applied a reverse time migration approach.  The closest work to this paper is the
thesis of one of us (F. Yang~\cite{FanPhD}) in which the use of the LSM to detect impenetrable objects in a waveguide
is analyzed and implemented.  {These results have not been published in an academic journal, and some}  of the results from the current paper are taken from the thesis with acknowledgement.

An outline of the paper is as follows.  In the next section (Sec.~\ref{forw}) we briefly discuss the forward problem. In Section \ref{subsec:forward:modal} we
give more details of the {well known} fundamental solution for an electromagnetic  waveguide, as well as an analysis of the
\emph{blocked waveguide} problem (or semi-infinite waveguide) in Section~\ref{wg:forward:half_pipe}.  
Existence and uniqueness for the forward problem is then summarized in Section~\ref{wg:forward:weak_vari:bdd}.
We then move on to the inverse problem in Section~\ref{wg:inverse}.  We start by recalling the dyadic Green's function for the standard scattering problem, and then in Section~\ref{wg:inverse:uniqueness}  give an uniqueness result for the  electromagnetic waveguide.  Section~\ref{wg:inverse:nearfieldop} is devoted to factoring the near field operator and deriving mapping properties, while Section~\ref{wg:inverse:lsm} presents the main result of the paper justifying the LSM for electromagnetic waveguides.  We then make some observations concerning the Generalized LSM (GLSM) in Section~\ref{glsm}. This is followed by our last major section, Section~\ref{Sec-Numresult}, in which we present some numerical results, and
we end with a brief conclusion and discussion (Section~\ref{Sec-Concl}).

Throughout this paper, we will distinguish vectors by means of boldface. Moreover, we will denote the divergence and the rotational of a regular enough vector field $\bfu$ by $\dv\bfu$ and $\cl\bfu$, respectively.

%
\section{Forward problem}\label{forw}

As discussed in the introduction, we consider an  infinite tubular waveguide $W$, generated by translates
of its cross-section $\Sigma\subseteq \mathbb{R}^{2}$. 
Note that the axis of the waveguide is parallel to the $x_3$-axis, and we can identify points in $\mathbb{R}^3=\mathbb{R}^2\times\mathbb{R}$ by writing $\mathbf{x}=(x_1,x_2,x_3)=(\hat{\mathbf{x}},x_3)$.
We then denote by $\boldsymbol{\nu} _0 := (\hat{\boldsymbol{0}},1)$, $\boldsymbol{\nu}:=(\hat{\boldsymbol{\nu}},0)$ and $\boldsymbol{\nu} _D $ the unit vector fields that are normal a.e. to $\Sigma_s:=\Sigma\times\{ s\}$ ($s\in\mathbb{R}$), $\Gamma :=\partial W=\partial\Sigma\times \mathbb{R}$ and $\partial D$, and directed to the right and outwards, respectively; see Figure \ref{fig:domain}.
We shall also denote by $\hx,\hy,\hz$ the standard unit vectors in $\bbR^3$.

\begin{figure}[h]
\begin{center}
\setlength{\unitlength}{.125mm}
\begin{picture}(950,200)
  \linethickness{0.3pt}
  \put(0,0){\line(1,0){950}}
  \put(0,200){\line(1,0){950}}
  \put(200,100){\oval(75,200)}
  \put(570,90){\oval(100,40)}
  \put(200,50){\vector(1,0){50}}
  \put(200,0){\vector(0,-1){50}} 
  \put(614,72){\vector(1,-1){35}} 
  \put(375,3){$\Gamma$}
	\put(195,100){$\Sigma_s$}
  \put(555,80){$D$}
  \put(225,60){$\boldsymbol{\nu }_0$}
  \put(210,-25){$\boldsymbol{\nu }$}
	\put(640,57){$\boldsymbol{\nu }_D$}
\end{picture}
\end{center}
\vspace*{.375cm}
\caption{A schematic 2d-view of the problem geometry:  The penetrable obstacle occupies an unknown region $D$ inside the waveguide $W$. 
}\label{fig:domain}
\end{figure}
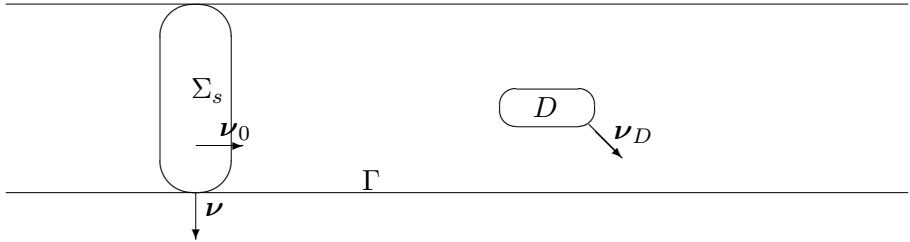

We assume that the waveguide is filled with air or vacuum; in particular, the background electric permittivity and magnetic permeability are
$$
\varepsilon _0=1/(\mu_0 c_0^2)\, Fm^{-1} \quad\mbox{ and } \quad \mu_0=4\pi 10^{-7}\, Hm^{-1} \, ,
$$
respectively. Here and in the sequel $c_0$ stands for the speed of light in vacuum.

Concerning the material which fills the scatterer $D$, we denote by  $\tilde\mu (\bfx )$, $\tilde\varepsilon (\bfx)$ and $\tilde\sigma (\bfx)$ its magnetic permeability, electric permittivity and conductivity, respectively. We assume that the magnetic permeability is constant, $\tilde\mu (\bfx )=\mu_0$, and that the conductivity is non-negative, $\tilde\sigma (\bfx)\geq 0$.
In general, $\tilde\varepsilon$ is a matrix function of position, but here we understand it as a scalar function (which is the case when the material is isotropic and uniform in all directions). The the anisotropic case  would involve no extra mathematical
difficulties. We further assume that it is piecewise smooth and bounded so there are constants $\tilde\varepsilon _0$
and $\tilde\varepsilon _1$ such that
\begin{equation*}
0<\tilde\varepsilon _0 < \tilde\varepsilon < \tilde\varepsilon _1 \quad\mbox{a.e. in } W \, .
\end{equation*}

Let us consider the time-harmonic case, and denote by $\omega$ and $k:=\omega\sqrt{\varepsilon_0\mu_0}$ the angular frequency and  the wavenumber for the background medium, respectively. We define the relative quantities $\D\varepsilon =  (\frac{\tilde\varepsilon}{\varepsilon_0} + i\frac{\tilde\sigma}{\omega\varepsilon_0} )$ and $\D\mu = \frac{\tilde\mu}{\mu_0}$; notice that $\varepsilon=\mu=1$ in the background $W\setminus\overline{D}$. Then, the time harmonic system of Maxwell's equations consists of finding the total electric and
magnetic fields denoted by the 
complex valued vector fields $\bfE\equiv \bfE (\bfx)$ and $\bfH\equiv \bfH (\bfx)${, that} satisfy  the following {Maxwell system} in a weak sense:
\begin{eqnarray}
    -ik\varepsilon \bfE - \cl \bfH &=&  \bfO  \mbox{ in }W\, , \label{time_harmonic_clH} \\[1ex]
    -ik\mu \bfH + \cl \bfE &=& \bfO \mbox{ in } W\, . \label{time_harmonic_clE}
\end{eqnarray}

Since we have assumed that the boundary of the waveguide is a perfect electric conductor (PEC, e.g. made of metal), then
\begin{equation}\label{fwd-total:bdry_shell}
    \boldsymbol{\nu}\times\bfE = \bfO \quad\mbox{ on } \Gamma \, .
\end{equation}
Using (\ref{time_harmonic_clE}) we can eliminate the unknown $\bfH$ from (\ref{time_harmonic_clH}) and rewrite the time harmonic Maxwell's system as a second order system of equations in terms of $\bfE$:
\begin{equation}
    \cl \cl\bfE - k^2\varepsilon\bfE = \bfO \quad\mbox{in } W \, , \label{fwd-total:maxwell_govern} 
\end{equation}
together with the boundary condition (\ref{fwd-total:bdry_shell}). Notice that we could similarly eliminate $\bfE$ and rewrite the system in terms of $\bfH$.

Let $\bfE^i$ denote a given incident field, which satisfies the Maxwell's system in the absence of the scatterer $D$:
\begin{eqnarray*}
    \cl\cl \bfE^i - k^2 \bfE^i = \boldsymbol{0}  \quad 
&\mbox{in } W_{(-R,R)} \, , \\[1ex]
    \boldsymbol{\nu}\times\bfE^i = \bfO &\mbox{on } \Gamma_{(-R,R)} \, ,
\end{eqnarray*}
where $W_{(-R,R)}=
\Sigma\times (-R,R)$ and $\Gamma_{(-R,R)}=
\partial\Sigma\times (-R,R)$, with $R>0$ big enough so that $\overline{D}\subset W_{(-R,R)}$. Of particular interest for this paper, we shall consider the incident field excited by an electric point source at $\bfy $ with polarization vector $\bfp \in\mathbb{R}^3$ ($|\bfp|=1$); typically, the sources will be located on a cross-section $\Sigma_{r}$ with $|r|>R$ (notice that then  $\bfy\in W\setminus \overline{D}$). 
This means that 
\begin{eqnarray*}
    \cl\cl \bfE^i - k^2 \bfE^i = \bfp\delta_{\bfy}  \quad 
&\mbox{in } W\, , \\[1ex]
    \boldsymbol{\nu}\times\bfE^i = \bfO &\mbox{on } \Gamma \, ,
\end{eqnarray*}
where $\delta_{\bfy}$ denotes the Dirac delta distribution centered at $\bfy$. 

The total field $\bfE$ is then decomposed into the incident and scattered fields, $\bfE=\bfE^i+\bfE^s$ in $W$. To close the system (\ref{fwd-total:maxwell_govern}-\ref{fwd-total:bdry_shell}) we need a suitable radiation condition on the scattered field $\bfE^s(\bfx)$ as $x_3 \to\pm\infty$ that imposes the physical restriction  that it  is \emph{an outgoing wave} in the waveguide. We will formalize this condition by using waveguide modes
in the next section.

%
\subsection{Modal Solutions to Maxwell's Equations in the Waveguide}\label{subsec:forward:modal}

The waveguide supports modes obtained from either of the following families (c.f. \cite{ChenToTai}, \cite[Appendix A]{FanPhD}):
\begin{itemize}
    \item The first family consists of the waveguide modes
        $$ \bfM_m (\bfx) = \cl \big(u_m(\hat{x})e^{ih_mx_3}\hz\big) =  \left(\!\begin{array}{c} \frac{\pd u_m}{\pd x_2} \\[1ex] -\frac{\pd u_m}{\pd x_1} \\[1ex] 0 \end{array} \!\right) e^{ih_mx_3} \quad\mbox{for } m=1,2,\ldots\, , $$
        where $h_m^2+\lambda_m^2=k^2$, and $\{(\lambda_m,u_m)\}_{m=0}^{\infty}$ are the eigenpairs of the Neumann problem for the negative surface Laplacian $-\dt_{\Sigma}$ on the cross-section of the waveguide. 
We sort the values $\{\lambda_m\}_{m=0}^{\infty}$ in ascending order, in particular $\lambda_m>0$ for $m=1,2,\ldots$ (notice that $\lambda_0=0$ is not considered because then $u_0$ is constant and $\bfM_0=\boldsymbol{0}$). We also rescale the eigenfunctions to have $\{u_m\}_{m=0}^{\infty}$ orthonormal in $L^2(\Sigma)$, in which case $\{ \frac{u_m}{\sqrt{\lambda_m+1}}\}_{m=0}^{\infty}$ defines an orthonormal basis of $H^1(\Sigma)$. 
    \item The second family is given by
        $$ \bfN_n (\bfx ) = \frac{1}{k}\,\cl\cl\big(v_n(\hat{x})e^{ig_nx_3}\hz\big) = \frac{1}{k}  \left(\!\begin{array}{c} ig_n  \frac{\pd v_n}{\pd x_1}\\[1ex] ig_n \frac{\pd v_n}{\pd x_2} \\[1ex] \mu_n^2 v_n  \end{array} \!\right) e^{ig_nx_3}  $$ 
for $n=1,2,\ldots\, ,$        where $k^2 = \mu_n^2 + g_n^2$, and $\{(\mu_n,v_n)\}_{n=1}^{\infty}$ is the set of eigenpairs of the Dirichlet problem for $-\dt_{\Sigma}$ on the cross-section of the waveguide. 
				Here again, we sort $\{\mu_n\}_{n=1}^{\infty}$ increasing and we rescale the eigenfunctions so that $\{v_n\}_{n=1}^{\infty}$ is orthonormal in $L^2(\Sigma)$, in which case $\{ \frac{v_n}{\sqrt{\mu_n}}\}_{n=1}^{\infty}$ is an orthonormal basis of $H^1_0(\Sigma)$. 
\end{itemize}

Notice that, since we have assumed that $\Sigma $ is convex, these eigenfunctions have the further regularity $u_m,v_n\in H^2(\Sigma)$; see \cite[Theorems 3.2.1.2-3.2.1.3]{GrisvardBook}.

Here and in the sequel we avoid the cut-off wavenumbers $k\in \{ \lambda_m\} _{m=1}^{\infty} \cup \{ \mu_n \} _{n=1}^{\infty}$.  With this assumption, $h_m=\sqrt{k^2-\lambda^2_m}\neq 0$ and $g_n=\sqrt{k^2-\mu^2_n}\neq 0$ for all $m,n=1,2,\ldots$, and we can define them by choosing the square root branch with non-negative real and imaginary parts.   

The behavior of the waveguide modes depend on {the coefficients $g_n$ and $h_m$}:
\begin{itemize}
    \item Modes for which $h_m$ (or $g_n$) are real are said to be \emph{traveling waves}. They satisfy a Sommerfeld type outgoing radiation condition along the axis of the waveguide; for example, for $x_3>0$,
        $$ \frac{\pd \bfM_m}{\pd x_3} - ih_m\bfM_m = \bfO \, . $$
    \item Modes for which $h_m$(or $g_m$) are purely imaginary are said to be \emph{evanescent}. They decay or grow exponentially along the axis of the waveguide; for example, for $x_3>0$, 
        $$ \bfM_m = \cl(u_me^{ih_m | x_3 |}) = \cl(u_me^{- | h_m | x_3}) \to\mathbf{0} \quad\mbox{ as }x_3 \to+\infty\, . $$
        \item {We assume that $k$ is not a cut-off frequency (also called a Rayleigh frequency) for the wave guide which implies
        that $h_m\not=0$ and $g_n\not =0$ for any $m,n$, so all modes are either evanescent or travelling.}
\end{itemize}
It is clear that, for a fixed wavenumber $k$, the number of traveling waves is bounded and the remaining modes are evanescent. In contrast to a sound hard acoustic waveguide, there may be no traveling modes if the wavenumber $k$ is too small.

The constant factor $1/k$ in the definition of $\bfN_n$ is convenient for the following relations:
\begin{equation}\label{wg:forward:modal:explicit-relation-M-N}
    \bfN_m = \frac1k\cl \bfM_m\qquad\mbox{and}\qquad \bfM_m = \frac1k\cl \bfN_m \qquad\forall m=1,2,\ldots \, . 
\end{equation}

For later use, let us consider a bounded section of the waveguide $W_{(s_1,s_2)}=\Sigma\times (s_1,s_2)$ and introduce the space
$$
{\bfX}_{(s_1,s_2)}\, = \, \{\bfv\in H(\mathbf{curl},W_{(s_1,s_2)}) ; \, \boldsymbol{\nu}\times\bfv=0\,\mbox{ on } \Gamma _{(s_1,s_2)} \} \, ,
$$
where $\Gamma _{(s_1,s_2)}=\partial\Sigma\times (s_1,s_2) $. Moreover, on any Lipschitz surface $S$ contained in $W_{(s_1,s_2)}$ with normal vector field $\boldsymbol{\nu}_S$, we consider $L^{2}_T(S) $ to be  the subspace of fields in $(L^2(S))^3$ tangential to $S$. Moreover,
  the standard dual space of $H^{1/2}(S) $ is denoted  $\widetilde{H}^{-1/2}(S) $. We also consider the following space of traces:
$$
\widetilde{H}^{-1/2}(\mathrm{div},S) \, = \, 
\{ \bff\in \widetilde{H}^{-1/2}(S)^3; \, \bff=\boldsymbol{\nu}_{S}\times\bfv|_{S} \,\mbox{ for some }\bfv\in {\bfX} _{(s_1,s_2)}\} \, ;
$$
and denote by $\widetilde{H}^{-1/2}(\mathbf{curl},S)$ its dual space.
%
In particular, when $S=\Sigma_s$ with $s\in [s_1,s_2]$, these spaces can be characterized in terms of modes:
$$
\begin{array}{c}
\bff\in\widetilde{H}^{-1/2}(\mathbf{curl},\Sigma_s) \quad \mbox{if, and only if,} \quad \displaystyle\sum _{m=1}^{\infty} |\alpha _m |^2 |\lambda_m | + \sum _{n=1}^{\infty} |\beta_n|^2 |\mu_n |^{3} <\infty \, ,
\\
\bff\in\widetilde{H}^{-1/2}(\mathrm{div},\Sigma_s)  \quad \mbox{if, and only if,} \quad \displaystyle\sum _{m=1}^{\infty} |\alpha _m |^2 | \lambda_m | ^{3} + \sum _{n=1}^{\infty} |\beta_n |^2 |\mu_n | <\infty \, ,
\end{array}
$$
for each $\bff=\displaystyle\sum _{m=1}^{\infty} \alpha _m \sgd u_m+ \sum _{n=1}^{\infty} \beta_n \sgd \times v_n\in \widetilde{H}^{-1/2}(\Sigma _s)^3$; 
indeed, the natural norms on $\widetilde{H}^{-1/2}(\mathbf{curl},\Sigma_s) $ and $\widetilde{H}^{-1/2}(\mathrm{div},\Sigma_s) $ are equivalent to
\begin{equation}
\begin{array}{c}
|| \bff ||_{\widetilde{H}^{-1/2}(\mathbf{curl},\Sigma_s)} =  ( \displaystyle\sum _{m=1}^{\infty} |\alpha _m |^2 |\lambda_m | + \sum _{n=1}^{\infty} |\beta_n|^2 |\mu_n |^{3}   )^{1/2} , 
\\
|| \bff ||_{\widetilde{H}^{-1/2}(\mathrm{div},\Sigma_s)} =  ( \displaystyle\sum _{m=1}^{\infty} |\alpha _m |^2 | \lambda_m | ^{3} + \sum _{n=1}^{\infty} |\beta_n |^2 |\mu_n |  ) ^{1/2} ,
\end{array}\label{norms}
\end{equation}
respectively, see  \cite[Paragraph 3.1.3.2]{FanPhD}. Moreover, for any $t\in\mathbb{R}$, the space $H^{t}_T(\Sigma_s)$ consists of tangential fields $\bff$ on $\Sigma_s$ such that 
$$
 \sum _{m=1}^{\infty} |\alpha _m |^2 |\lambda_m |^{2(t+1)} + \sum _{n=1}^{\infty} |\beta_n | ^2 |\mu_n | ^{2(t+1)} <+\infty \, ,
$$
and may be endowed with the norm 
$$
|| \bff ||^2_{H^{t}_T(\Sigma_s)} = \sum _{m=1}^{\infty} |\alpha _m |^2 |\lambda_m | ^{2 (t+1)} + \sum _{n=1}^{\infty} |\beta_n |^2 |\mu_n| ^{2 (t+1)}\, .
$$

\subsection{The Blocked Waveguide and the Dirichlet-to-Neuman Map}\label{wg:forward:half_pipe}

The radiation condition, which is yet to be defined, must {constrain} any scattered field  to be an \emph{outgoing wave} in the waveguide: when represented by the waveguide modes, each contributing mode must either propagate outwards or decay exponentially away from the scatterer. With this in mind, let us consider a solution $\bfU$ of Maxwell's system (\ref{fwd-total:maxwell_govern}-\ref{fwd-total:bdry_shell})  in an unbounded section of the waveguide 
of the form $W_I=\Sigma\times I$ where $I=(-\infty ,s)$ or $(s,\infty)$ with $s\in\mathbb{R}$. 

For $I=(s,\infty)$, we say that $\bfU$ satisfies the outgoing radiation condition (ORC) if $\bfU \in H_{loc}(\mathbf{curl} ,W_I) $ and, for $|x_3|$ big enough, it can be written in terms of the waveguide modes as
$$ \bfU = \sum_{m=1}^{\infty}A_m\bfM_m + \sum_{n=1}^{\infty}B_n\bfN_n \, .  $$


The following lemma is shown in \cite[Lemma 3.1.3]{FanPhD} and states the well-posedness of the blocked (or semi-infinite) waveguide problem  in the absence of any scatterer.
 \begin{lemma}\label{fwd:half_pipe:unbdd} Given $\bfQ\in\widetilde{H}^{-1/2}(\mathrm{div},\Sigma_s)$, there exists a unique solution $\bfU\in H_{loc}(\tcl,W_{(s,\infty)})$ to the following problem:
\begin{equation} \label{fwd:half_pipe:unbdd-problem}
\begin{array}{rl}
   \cl\cl\bfU - k^2\bfU = \bfO & \mbox{in } W_{(s,\infty)}\, , \\[1ex]
   \boldsymbol{\nu}\times\bfU = \bfO & \mbox{on }  \Gamma_{(s,\infty)}=\partial\Sigma\times (s,+\infty)\, , \\[1ex]
   \nc\times\bfU = \bfQ & \mbox{on }  \Sigma_s\, , \\[1ex]
   \mbox{$\bfU$ satisfies the ORC} & \mbox{for }  x_3 \to +\infty \, . 
\end{array}
\end{equation}
Moreover, the solution has the expansion
\begin{equation} \bfU = \sum_{m=1}^{\infty} A_m \bfM_m + \sum_{n=1}^{\infty} B_n \bfN_n \, , \label{modal}
\end{equation}
when $\bfQ = \displaystyle\sum_{m=1}^{\infty} A_m \sgd u_m - \frac{i}{k} \sum_{n=1}^{\infty} B_n g_n \sgd \times v_n$ on $\Sigma_s$. 

The same result holds for $W_{(-\infty,s)}$ if $h_m$ and $g_n$ are replaced by $-h_m$ and $-g_n$ in the expansions.
\end{lemma}

%
We can use this result to define an important operator for our upcoming analysis,  denoted by $T^{\pm}_{s}$ and which is the analogue of the Dirichlet-to-Neumann (DtN) map for Helmholtz equation. Specifically, for some fixed $s\in\mathbb{R}$ and  any tangential field $\bfQ\in\widetilde{H}^{-1/2}(\mathrm{div},\Sigma_s)$, we take
\begin{eqnarray}
T^+_{s}(\bfQ) = \nc\times(\cl\bfU)|_{\Sigma_s} \, , \label{fwd-dtn-general-form}
\end{eqnarray}
where $\bfU$ solves (\ref{fwd:half_pipe:unbdd-problem}) in $W_{(s,\infty)}$. 
A similar operator can be defined by considering the analogue on $W_{(-\infty,s)}$, and we  identify the operator on each specific cross-section by means of a superscript: $T^{+}_{s}$ and $T^{-}_{s}$ on $\Sigma_s$ when using $W_{(s,\infty )}$ and $W_{(-\infty,s)}$, respectively. The analysis of the two operators  $T^{\pm}_{s}$ is analogous and, accordingly, we only give details for $T^+_{s}$. 

To derive a series representation of $T^+_{s}$ in terms of the waveguide modes, we can make use of the explicit form of the solution $\bfU$ provided in Lemma \ref{fwd:half_pipe:unbdd}:
\begin{equation} \label{fwd-dtn-explicit-form}
\begin{array}{rl}
    T^{+}_{s}(\nc\times \bfU|_{\Sigma}) \, =&\! -\, i\,\displaystyle\sum_{m=1}^{\infty} \Big\la \nc\times \bfU|_{\Sigma_s}, \left(\!\!\begin{array}{c} \sgd u_m \\ 0 \end{array}\!\!\right) \!\Big\ra_{\!\Sigma_s}\, \frac{h_m}{\lambda_m^2} \left(\! \!\begin{array}{c} \vgd u_m \\ 0\end{array}\!\!\right) \nonumber \\
    & \, +\, i\,k^2\, \displaystyle\sum_{n=1}^{\infty} \Big\la \nc\times \bfU|_{\Sigma_s}, \left(\!\!\begin{array}{c} \vgd v_n \\ 0 \end{array}\!\!\right) \!\Big\ra_{\!\Sigma_s} \, \frac{1}{g_n\mu_n^2}\! \left(\!\!\begin{array}{c} \sgd v_n \\ 0 \end{array}\!\!\right)\! .
\end{array}
\end{equation}
This expression is explicitly derived in \cite[Section 3.1.3.5]{FanPhD} and then used to deduce the following properties of the operator $T^{+}_{s}$ {using the  characterization of the norms in(\ref{norms})} (see \cite[Lemmas 3.1.4 and 3.1.5]{FanPhD}).
 \begin{lemma}\label{fwd-dtn-bdd&analytic} 
The operator $T^{+}_{s}$ is bounded  from $\widetilde{H}^{-1/2}(\mathrm{div},\Sigma_s)$ to $\widetilde{H}^{-1/2}(\mathrm{div},\Sigma_s)$. Moreover, there exists a neighborhood $\mathfrak{B}\subset \mathbb{C}$ of $k$ where $T^{+}_{s}$ depends analytically on $k$.
\end{lemma}

Let us notice that, by means of $T^{+}_{s}$, the blocked waveguide problem (\ref{fwd:half_pipe:unbdd-problem}) can be rewritten  in a bounded section of waveguide. Accordingly, the following result is the counterpart of Lemma~\ref{fwd:half_pipe:unbdd}, and we refer to \cite[Corollary 3.1.1]{FanPhD} for more details.
\begin{corollary}\label{fwd:half_pipe:bdd} For $s_1<s_2$ and $\bfQ\in\widetilde{H}^{-1/2}(\mathrm{div},\Sigma_{s_1})$, there exists a unique $\bfU\in H(\tcl,W_{(s_1,s_2)})$ such that 
\begin{equation}
\begin{array}{rcl}
   \cl\cl\bfU - k^2\bfU = \bfO & \mbox{in} & W_{(s_1,s_2)} \, , \\[1ex]
   \boldsymbol{\nu}\times\bfU = \bfO & \mbox{on} & \Gamma_{(s_1,s_2)}\, , \\[1ex]
   \nc\times\bfU = \bfQ & \mbox{on} & \Sigma_{s_1}\, , \\[1ex]
   \nc\times(\cl\bfU) = T^{+}_{s_2}(\nc\times\bfU) & \mbox{on} & \Sigma_{s_2} \, . \label{fwd:half_pipe:bdd-problem}
\end{array}
\end{equation}
\end{corollary}

%
%
\subsection{Analysis of the Forward Problem}\label{wg:forward:weak_vari:bdd}

Now we have all the tools we need both to impose a suitable radiation condition on the scattered field and to analyze the forward problem (\ref{fwd-total:bdry_shell})-(\ref{fwd-total:maxwell_govern}) closed with such a radiation condition. 
Let the scatterer $D$ be illuminated by a point source at $\bfy\in W$ located sufficiently far below $D$.  {By this we mean that we can choose} $R>0$ such that $\overline{D}\subset\Sigma\times W_{(-R,R)}$ and $y_3<R$. We then write the forward problem
as the equivalent problem of finding the total field $\bfE\in H({\rm\bf{}curl},W_{(-R,R)})$ such that 
\begin{equation}\label{fwd-total:problem_truncated}
\begin{array}{rcl}
   \cl\cl\bfE - k^2\varepsilon\bfE = \bf0 & \mbox{in} &W_{(-R,R)}\, , \\[1ex]
   \boldsymbol{\nu}\times\bfE = \bfO & \mbox{on} & \Gamma_{(-R,R)}\, , \\[1ex]
 \pm\nc\times(\cl (\bfE-\bfE^i)) = T^{\pm}_{R}(\nc\times(\bfE-\bfE^i)) & \mbox{on} & \Sigma_{\pm R} \, .
\end{array}
\end{equation}
This problem can be equivalently rewritten in terms of the scattered field as the problem of finding 
$\bfE^s\in H({\rm\bf{}curl},W_{(-R,R)})$ such that
\begin{equation}\label{fwd-sc:problem_truncated}
\begin{array}{rcl}
   \cl\cl\bfE^s - k^2\varepsilon\bfE^s = k^2(\varepsilon - 1) \bfE^i& \mbox{in} & W_{(-R,R)}\, , \\[1ex]
   \boldsymbol{\nu}\times\bfE^s = \bfO & \mbox{on} & \Gamma_{(-R,R)}\, , \\[1ex]
   \pm\nc\times(\cl\bfE^s) = T^{\pm }_{R}(\nc\times\bfE^s) & \mbox{on} & \Sigma_{\pm R} \, . 
\end{array}
\end{equation}

In order to write the scattering problem  in weak form, we define
the trace operator $\boldsymbol{\gamma} _T:H(\mathbf{curl},W_{(-R,R)})\to H^{-1/2}( \mathbf{curl} ,\Sigma_{ \pm R})$ for smooth vector fields on $W_{(-R,R)}$ by
 $\boldsymbol{\gamma}_T\bfv =\nc\times (\bfv |_{\Sigma_{ \pm R}} \times\nc)$.  We also denote by $(\cdot, \cdot )_{W_{(-R,R)}}$ the inner product in $L^2(W_{(-R,R)})^3$, and by $\langle \cdot , \cdot \rangle_{\Sigma_{ \pm R}}$ the duality product in $\widetilde{H}^{-1/2}(\mathrm{div},\Sigma_{\pm R})\times \widetilde{H}^{-1/2}(\mathbf{curl},\Sigma_{\pm R})$ so that: 
$$
\begin{array}{l}
\displaystyle (\bfu,\bfv)_{W_{(-R,R)}} = \int_{W_{(-R,R)} }\bfu\cdot\overline{\bfv}\, d\bfx \quad\forall \bfu,\bfv\in L^2(W_{(-R,R)})^3\, ,
\\
\displaystyle \langle\bfu,\bfv\rangle_{\Sigma _{\pm R}} = \int_{\Sigma _{ \pm R}} \bfu\cdot\overline{\bfv}\, dS \quad\forall \bfu,\bfv\in L^2(\Sigma _{ \pm R})^3\, .
\end{array}
$$
Formally, multiplying the first equation of (\ref{fwd-total:problem_truncated}) by the complex conjugate of a smooth test function $\bfv\in C^\infty(W_{(-R,R)})^3$ and applying Green's identity, we have that $\bfE\in H({\rm\bf{}curl},W_{(-R,R)})$
satisfies
\begin{equation}\label{fwd-total:weak_form}
\begin{array}{l}
\displaystyle ( \cl\bfE,\cl {\bfv} )_{W_{(-R,R)}}\! - k^2 (\varepsilon\bfE, {\bfv} )_{W_{(-R,R)}}  \!
+ \langle T^{\pm }_{R}(\nc\times\bfE) ,\boldsymbol{\gamma}_T {\bfv} \rangle_{\Sigma_{\pm R}} =
\mathcal{F}({\bfv}) \, .
\end{array}
\end{equation}
Here the antilinear functional
$$
\mathcal{F}(\bfv) = 
 \langle ( T^{\pm }_{ R}(\nc\times\bfE^i) \mp\nc\times (\cl\bfE^i ) ) , \boldsymbol{\gamma}_T {\bfv}  \rangle _{\Sigma_{ \pm R}}
\, .
$$
A weak formulation of the scattered field problem (\ref{fwd-sc:problem_truncated}) is formulated in the same way.

Given an incident field $\bfE^i $, the total field $\bfE\in\bfX_{(-R,R)}$ is a weak solution to the forward scattering problem if it satisfies (\ref{fwd-total:weak_form}) for all $\bfv\in\bfX_{(-R,R)}$. Problem (\ref{fwd-total:weak_form}) can then be analyzed using the analytic Fredholm theory to prove the following result.
\begin{theorem} \label{fwd:vari_well-posedness}
If $\tilde{\sigma}\geq\sigma_0>0$ in $D$ (or in some open bounded subdomain of $D$ with non-zero measure), then the forward scattering problem (\ref{fwd-total:weak_form}) is well-posed for any real wavenumber $k$. If $\tilde{\sigma}=0$ in $D$, then the problem is well-posed except for, at most, a discrete set of real $k$ values whose only possible accumulation point is $\infty$.
\end{theorem}
\begin{remark}
In  the remainder of the paper we assume that $k$ is such that the forward problem is well-posed.
\end{remark}

\begin{proof}
We now give a sketch of the proof of this result.  For more details of a similar argument in the PEC case see
\cite[Section 3.2]{FanPhD}: 
First note that the space $\bfX_{(-R,R)}$ admits the Helmholtz decomposition $\bfX_{(-R,R)}= \gd S_R \oplus \bfX^+_{(-R,R)} $, where
$$
\begin{array}{l}
S_{(-R,R)}= \{ p\in H^1(W_{(-R,R)}) ; \, p=0 \,\mbox{on } \Gamma _{(-R,R)} \} \, ,\\[1ex]
\bfX_{(-R,R)}^+=\big\{ \bfv^+\!\!\in\! \bfX_{(-R,R)} ; \, k^2 ( \varepsilon\bfv^+ ,\gd q)_{W_{(-R,R)}} \\[.5ex]
\hspace*{5cm} = \langle T^{\pm }_{R}(\nc\times\bfv^+) , \boldsymbol{\gamma}_T (\gd q)\rangle _{\Sigma _{\pm R}} \forall  q\in S_{(-R,R)} \big\} \, ; 
\end{array}
$$
see \cite[Lemma 3.2.2]{FanPhD} for a similar result.  Then
we can 
write $\bfE=\gd p+\bfE^+$ with $p\in S_{(-R,R)}$ and $\bfE^+\in\bfX_{(-R,R)}^+$. The problem for $p$ decouples by taking the test function in (\ref{fwd-total:weak_form}) with the form $\bfv=\nabla q$, and  $p\in S_{(-R,R)}$ must solve 
		\begin{equation*}
		k^2  ( \varepsilon\gd p ,\gd q )_{W_{(-R,R)}} -\langle T^{\pm }_{R}(\nc\times\gd p) , \boldsymbol{\gamma}_T ({\gd q}) \rangle _{\Sigma _{\pm R}} 
		=
		\mathcal{F} (\gd q)
		\quad\forall  q\in S_{(-R,R)} \, .
		\end{equation*}
This auxiliary problem satisfies the hypotheses of the Lax-Milgram lemma (see \cite[Lemma $3.2.1$]{FanPhD}) and, in consequence $p\in S_{(-R,R)}$ is uniquely determined.

Using the function $p$ from the previous step, the forward scattering problem can be rewritten in terms of $\bfE^+\in\bfX_{(-R,R)}^+$ as follows:
		\begin{equation}\label{fwd-total:weak_form+}
\begin{array}{l}
\displaystyle\big( \cl\bfE^+,\cl {\bfv}^+\big)_{W_{(-R,R)}} - k^2\big(\varepsilon\bfE^+, {\bfv}^+\big)_{W_{(-R,R)}} 
+\, \big\langle T^{ \pm }_{R}(\nc\times\bfE^+) ,\boldsymbol{\gamma}_T {\bfv}^+ \big\rangle_{\Sigma_{\pm R}} \, = \,
\tilde{\mathcal{F}}(\bfv^+)
\end{array}
\end{equation}
for all $\bfv^+\in\bfX_{(-R,R)}^+$, where $\displaystyle \tilde{\mathcal{F}}(\bfv^+) 
=
 \mathcal{F}(\bfv^+)
+ k^2  \big(\varepsilon \gd p ,{\bfv}^+\big)_{W_{(-R,R)}}
-  \big\langle  T^{\pm}_{ R}(\nc\times\gd p) , \boldsymbol{\gamma}_T {\bfv}^+  \big\rangle _{\Sigma_{\pm R}}$. 

Notice that
\begin{itemize}
\item the sesquilinear form $a_k^+:\bfX_{(-R,R)}^+\times \bfX_{(-R,R)}^+\to \mathbb{C}$ defined by
\begin{equation*}
\begin{array}{l}
a_k(\bfu^+,\bfv^+)\, = \,\displaystyle ( \cl\bfu^+,\cl {\bfv}^+ )_{W_{(-R,R)}} + k^2 (\varepsilon\bfu^+, {\bfv}^+ )_{W_{(-R,R)}}
+ \langle T^{\pm ,0}_{R}(\nc\times\bfu^+) ,\boldsymbol{\gamma}_T {\bfv}^+ \rangle_{\Sigma_{\pm R}} 
\, ,
\end{array}
\end{equation*}
is bounded and coercive;
\item $\bfX_{(-R,R)}^+$ is compactly embedded in $L^2(W_{(-R,R)})^3$; 
 \item the DtN map $T^{ \pm }_{R} : \widetilde{H}^{-1/2}(\mathrm{div},\Sigma_{\pm R}) \to \widetilde{H}^{-1/2}(\mathrm{div},\Sigma_{\pm R}) $ is the superposition of a positive operator $T^{\pm ,0 }_{R}$ and a compact map $T^{\pm ,c }_{R}$ (see \cite[Lemma 3.2.4]{FanPhD}).
\end{itemize}
This allows us to rewrite  (\ref{fwd-total:weak_form+}) in operator form as  \[
 (I + B_k) \,\bfE^+ = \bff,
 \] where the operator $B_k:\bfX_{(-R,R)}^+\to \bfX^+_{(-R,R)}$ is compact and analytic with respect to $k$ in a suitable subdomain of the complex plane containing the real line (removing the cut-off frequencies, see paragraph \ref{subsec:forward:modal}).  Now we can see that
 this operator equation   admits at most one solution $\bfE^+\in \bfX_{(-R,R)}^+$ in the following two cases:
	\begin{itemize} 
	\item for any real wavenumber $k\in\mathbb{R}$, if $\tilde{\sigma}\geq\sigma_0>0$ in $D$ (or in some open bounded subdomain of $D$ with non-zero measure);
	\item for any pure imaginary $k = ic$ with $c>0$ small enough, if $\tilde{\sigma}=0$ in $D$;
	\end{itemize}
see the proof of \cite[Theorem 3.2.1]{FanPhD} for an analogous result.

Then, by applying the analytic Fredholm theory~\cite{ColtonKress-IP}, we conclude that the forward problem is well-posed except for (at most) a countable set of real wavenumbers.
\end{proof}

%
\section{Inverse Problem}\label{wg:inverse}

In this section, we shall provide a theoretical basis for the Linear Sampling Method (LSM)  approach to the inverse problem in the waveguide geometry. There are two important results here: the uniqueness of the solution of the inverse problem, and the justification of the LSM for the reconstruction of the shape of scatterer.
To this end, we recall some results about the background Green's function in the waveguide. It is well-known that, for electromagnetic waves, the Green's functions are dyadic functions (second order tensors that can be written as $3\times3$ matrices) with appropriate boundary conditions on $\Gamma$. Specifically, we here consider the electric Green's function $\bbG_{e}$ which satisfies a PEC condition on the boundary $\Gamma$ of the waveguide:
\begin{eqnarray*}
\clx\clx\bbG_{e}(\bfx,\bfy) - k^2\bbG_{e}(\bfx,\bfy) \, = \,\delta_{\bfx-\bfy}  \, \bbI \quad\mbox{in } W\, ,\\[1ex] 
\boldsymbol{\nu}\times\bbG_{e}(\bfx,\bfy) \, = \, \bfO\quad\mbox{on }\Gamma \, , 
\end{eqnarray*}
where  $\bfy\in W$ represents the point source and $\bfx\in W$ any evaluation point. Moreover, $\bbI$ is the identity matrix, whereas $\mathbf{a}\times\mathbb{B}$ and $\cl\mathbb{B}$ denote the matrices whose $l$-th columns are $\mathbf{a}\times\mathbf{b}_l$ and $\cl \mathbf{b}_l$, respectively, for  any column vector function $\mathbf{a}$ and any dyadic function $\mathbb{B}$ (written as a matrix with columns $\mathbf{b}_l$). In particular,
an incident wave due to a point source at $\bfz\in W_R\setminus\overline{D}$ with polarization $\bfp\ ( |\bfp |\not=0)$ is given by $\bfu^i(\bfx;\bfz,\bfp)=\bbG_e(\bfx,\bfz)\bfp$; and we recall that $\bfu^s (\bfx;\bfz,\bfp)$ stands for its associated scattered field (that is, the solution of the forward problem (\ref{fwd-sc:problem_truncated}) for such an incident wave), and  similar notation is used for the associated total field $ \bfu (\bfx;\bfz,\bfp) = \bfu^s (\bfx;\bfz,\bfp ) + \bfu^i(\bfx;\bfz,\bfp)$. %

We wish to compare the above electric Green's function in the waveguide with the one in free space, which solves
\begin{equation*}
\clx\clx\bbG_{0}(\bfx,\bfy) - k^2\bbG_{}(\bfx,\bfy) \, = \,\delta_{\bfx-\bfy}  \, \bbI \quad\mbox{in } \mathbb{R}^3\, ,\
\end{equation*}
together with a tensor form of the Silver-Muller radiation condition. This can be written explicitly as
\begin{eqnarray}
\bbG_0(\bfx,\bfy) = \Phi(\bfx,\bfy)\bbI + \frac{1}{k^2}\gd_\bfy\gd_\bfy\Phi(\bfx,\bfy)\quad \forall \bfx,\bfy\in\mathbb{R}^3 \mbox{ with }\bfx\ne\bfy \, , \label{inv-green-free-space}
\end{eqnarray}
where $\Phi$ stands for the fundamental solution of Helmholtz equation:
$$
\D \Phi (\bfx,\bfy) = \frac{\exp(ik |\bfx-\bfy |)}{4\pi |\bfx-\bfy |} \quad \forall \bfx,\bfy\in\mathbb{R}^3 \mbox{ with }\bfx\ne\bfy\, ,
$$
and  $\gd_\bfy\gd_\bfy\Phi(\bfx,\bfy)$ denotes its Hessian matrix.
The following result is proven in \cite[Lemma 3.3.1]{FanPhD} and provides a decomposition of Green's function in the waveguide in terms of that in free-space. This provides a clear statement of the singularity present in the 
waveguide fundamental solution {(the result in \cite{FanPhD} was proved when $\Sigma$ has a smooth boundary; here the regularity of the remainder term is determined by the cross section $\Sigma$).}
 \begin{lemma}\label{inv:green-decompose}  In any bounded segment $W_{(s_1,s_2)}$ of the waveguide $W$, the electric type dyadic Green's function for the waveguide can be decomposed into $\bbG_e = \bbG_0 + \bbJ$, 
where 
the 
dyadic function $\bbJ$ is {in $H^1_{\rm{}loc}$.}
\end{lemma}

\subsection{Uniqueness Result}\label{wg:inverse:uniqueness}
In order to prove the uniqueness result for the inverse problem, we start by stating two lemmas: The first gives us a \emph{representation formula} of the solution of the forward problem in the waveguide excluding the scatterer $D$; this result is analogous to the well-known Stratton-Chu formula in free-space and is proven in \cite[Lemma 3.3.2]{FanPhD}.
 \begin{lemma}\label{inv:uniqueness:rep_fml}(Representation Formula) 
Let $\bfU$ be a solution of the Maxwell's system 
\begin{equation*}
\begin{array}{ll}
\cl\cl\bfU- k^2 \bfU \, =  \, \bfO & \quad\mbox{in } W\setminus\overline{D}\, ,\\[1ex]
\boldsymbol{\nu}\times\bfU \, =\, \bfO &\quad\mbox{on }\Gamma \, , \\[1ex]
\mbox{$\bfU$ satisfies the ORC} & \quad\mbox{for }  \vert x_3\vert \to +\infty \, .
\end{array}
\end{equation*}
Then 
$$ 
\bfU(\bfx) 
= \int_{\pd {D} } \!\!\! \Big(
\big( \boldsymbol{\nu}_{D}\times (\cl\bfU(\bfy)) \big)\cdot\bbG_e(\bfx,\bfy)
\,-\, \bfU(\bfy)\cdot\big( \boldsymbol{\nu}_{D}\times (\cly\bbG_e(\bfx,\bfy)) \big) 
 \Big) dS_{\bfy} \, . 
$$
\end{lemma}
\begin{remark}\begin{enumerate}\item
Above, the dot-product is understood as vector-matrix or matrix-vector multiplication depending on the position of dyadic and vector functions. The next result states that the scattered field satisfies the standard reciprocity relation.
\item This result is proved in the same way as the corresponding Stratton-Chu formula in free-space~\cite{Kirsch-Hettlich}.
\end{enumerate}
\end{remark}
 \begin{lemma}\label{inv:uniqueness:reciprocity}(Reciprocity relation) 
For any points $\bfx,\bfz\in W\backslash\overline{D}$ and polarization vectors $\bfp,\bfq\in\mathbb{R}^3$ ($ |\bfp |\not=0,\;  |\bfq |\not=0$), it holds
 $$ \bfu^s(\bfx;\bfz,\bfp)\cdot\bfq = \bfu^s(\bfz;\bfx,\bfq)\cdot\bfp\, . $$
\end{lemma}
\begin{remark} This result also holds for a bounded perfectly conducting scatterer (incorrectly stated in \cite[Lemma 3.3.3]{FanPhD}).  The following proof is closely related to the proof of \cite[Lemma 3.3.3]{FanPhD}, but correcting the
statement of the result.
\end{remark}

\begin{proof}
By using the representation formula in Lemma~\ref{inv:uniqueness:rep_fml}, we have
\begin{equation*}
\begin{array}{l}
    \bfu^s(\bfx;\bfz,\bfp) = \displaystyle\int_{\pd D} \!\Big( -\bfu^s (\bfy;\bfz,\bfp )\cdot \big( \nd\times(\cly\bbG_e(\bfx,\bfy) ) \big) 
+  \big( \nd\times(\cl\bfu^s(\bfy;\bfz,\bfp)) \big) \cdot\bbG_e(\bfx,\bfy)\Big) \, dS_{\bfy}\, .
\end{array}
\end{equation*}
In addition,  by using Green's second identity and  the properties of the electric type dyadic Green's function,
\begin{equation*}
\begin{array}{l}
    \boldsymbol{0} = \displaystyle\int_{\pd D} \!\Big( -\bbG_e (\bfy,\bfz)\bfp \cdot \big( \nd\times(\cly\bbG_e(\bfx,\bfy) ) \big) 
+  \big( \nd\times(\cl\bbG_e(\bfy,\bfz)\bfp) \big) \cdot\bbG_e(\bfx,\bfy)\Big) \, dS_{\bfy}\, ; 
\end{array}
\end{equation*}
see \cite[eq. (3.48)]{FanPhD}. Adding both results leads to
\begin{equation*}
\begin{array}{l}
    \bfu^s(\bfx;\bfz,\bfp) =\displaystyle \int_{\pd D} \Big(- \bfu (\bfy;\bfz,\bfp )\cdot \big( \nd\times(\cly\bbG_e(\bfx,\bfy) ) \big)
 +  \big( \nd\times(\cl \bfu (\bfy;\bfz,\bfp)) \big) \cdot\bbG_e(\bfx,\bfy)\Big) \, dS_{\bfy}\, ;
\end{array}
\end{equation*}
which, making use of dyadic identities and the second vector-dyadic Green's identity, can be  to rewritten as
\begin{equation}\label{r_c1tilde}
\begin{array}{l}
    \bfu^s(\bfx;\bfz,\bfp) =\displaystyle \int_{D} \Big( \big(\cl\cl \bfu  (\bfy;\bfz,\bfp )\big) \cdot \bbG_e(\bfx,\bfy) )
- \bfu (\bfy;\bfz,\bfp) \cdot \big(\cly\cly\bbG_e(\bfx,\bfy)\big) \Big) \, d\bfy\,  .
\end{array}
\end{equation}
Similarly, one can show that
\begin{equation*}
\begin{array}{l}
    \bfu^s(\bfz;\bfx,\bfq) \cdot \bfp
		=\displaystyle \int_{\pd D} \Big( -{\bfu}^s (\bfy;\bfx,\bfq )\cdot \big( \nd\times(\cl {\bfu} (\bfy;\bfz,\bfp ) ) \big) \Big. \\
  \hspace*{3.5cm}  \Big. +  \big( \nd\times(\cl {\bfu}^s(\bfy;\bfx,\bfq)) \big) \cdot {\bfu} (\bfy;\bfz,\bfp )\Big) \, dS_{\bfy}\, , 
\end{array}
\end{equation*}
and deduce that
\begin{equation}\label{r_c1tildebis}
\begin{array}{l}
    \bfu^s(\bfz;\bfx,\bfq) \cdot \bfp
		= \displaystyle\int_{D} \Big( \big(\cly\cly\bfu^s (\bfy;\bfx,\bfq )\big)\cdot \bfu (\bfy;\bfz,\bfp ) \Big.  \\
    \hspace*{3.5cm}  \Big. - {\bfu}^s(\bfy;\bfx,\bfq)) ) \cdot \big(\cl\cl \bfu (\bfy;\bfz,\bfp )\big) \Big) \, d\bfy \, .
\end{array}
\end{equation}
In consequence, from (\ref{r_c1tilde}) and (\ref{r_c1tildebis}) we have
\begin{equation*}
\begin{array}{l}
    \bfu^s(\bfx;\bfz,\bfp) \cdot\bfq  - \bfu^s(\bfz;\bfx,\bfq) \cdot \bfp \, =  \\
    \qquad =\, \displaystyle\int_{D} \!\Big( \big(\cl\cl \bfu  (\bfy;\bfz,\bfp )\big) \cdot  \bfu  (\bfy;\bfx,\bfq )
-  \bfu  (\bfy;\bfz,\bfp ) \cdot  \big(\cl\cl \bfu  (\bfy;\bfx,\bfq )\big)
		\Big) d\bfy 		\, , 
\end{array}
\end{equation*}
from which the result follows by noticing that  $ \cl\cl \bfu (\bfy;\bfz,\bfp ) = \varepsilon k^2 \bfu (\bfy;\bfz,\bfp ) $ and  $ \cl\cl \bfu (\bfy;\bfx,\bfq ) = \varepsilon k^2 \bfu (\bfy;\bfx,\bfq ) $ in $D$. 
\end{proof}

The above reciprocity relation is the basic tool we need to prove the uniqueness of solution for the inverse problem we are dealing with.  We state the result in slightly more generality than is needed for this paper.  More precisely, let us consider  two cross-sections $\Sigma_s$ and $\Sigma_r$ located on the same side {(i.e. both above or both below)} of the scatterer and with the receivers not further from the target than the sources; for the sake of simplicity, in the sequel we assume that both surfaces are below the scatterer, that is, $D\subseteq W_{(r,R)}$ and $s\leq r$. 
\begin{theorem}\label{inv:uniqueness} Let $D_1$ and $D_2$ be two penetrable scatterers completely contained in the waveguide and away from its boundary
, and with relative electric permittivities $\varepsilon_1$ and $\varepsilon_2$, respectively. 
If 
the tangential components of the fields $\bfu_1^s(\cdot;\bfy,\bfp)$ and $\bfu_2^s(\cdot;\bfy,\bfp)$, scattered by $D_1$ and $D_2$ respectively, coincide on $\Sigma_r$ for all 
sources $\bfy\in\Sigma_s$ and polarizations $\bfp\in\mathbb{R}^3$ ($ |\bfp |=1$), then  $D_1=
D_2$.
\end{theorem}
\begin{proof}
The following proof is a modification of that of \cite[Theorem 5.6]{ColtonKress-IP}. 
Let us suppose that $\bfw(\cdot;\bfx_0,\bfp) := \bfu_1^s(\cdot;\bfx_0,\bfp) - \bfu_2^s(\cdot;\bfx_0,\bfp)$ has null tangential component $\nc\times\bfw (\cdot;\bfx_0,\bfp) = \bfO$ on $\Sigma_r$. Then, by Lemma~\ref{fwd:half_pipe:unbdd}, $\bfw (\cdot;\bfx_0,\bfp)=\bfO$ in $W_{(-\infty,r)}$; and, by the unique continuation principle (see \cite[Theorem 4.13 and Remark 4.14]{MonkBook}), also $\bfw (\cdot;\bfx_0,\bfp) = \bfO$ in $\Omega= W\backslash\overline{(D_1\cup D_2)}$. Applying the reciprocity relation (see Lemma~\ref{inv:uniqueness:reciprocity} above), from  $\bfu_1^s(\bfy ;\bfx_0,\bfp) = \bfu_2^s(\bfy ;\bfx_0,\bfp)$ for $\bfy\in\Omega$ and $\bfx_0\in\Sigma_s$ we deduce that
$$
\bfu_1^s(\bfx_0;\bfy,\bfp) = \bfu_2^s(\bfx_0;\bfy,\bfp) \quad \forall \bfx_0\in\Sigma_s \, , \; \bfy\in\Omega\, , \;\bfp\in\mathbb{R}^3 \, .
$$
In case $D_1\ne D_2$, without loss of generality we may consider some point $\bfy^*\in\pd D_1$ such that $\bfy^*\not\in\overline{D_2}$ and use it to build the sequence
$$ \bfy_n^* = \bfy^* + \frac1n\,\boldsymbol{\nu}_{D_1}^*  \quad\mbox{for } n = 1,2,\ldots \, ,$$
where $\boldsymbol{\nu}_{D_1}^*$ is the unit outward normal to $\pd D_1$ at $\bfy^*$. Notice that $\bfy_n^*\in\Omega$ for $n$ big enough, so that we have already shown 
$$
\bfu_1^s(\bfx_0;\bfy_n^*,\bfp) = \bfu_2^s(\bfx_0;\bfy_n^*,\bfp) \quad \forall \bfx_0\in\Sigma_s\, ,\; \bfp\in\mathbb{R}^3 \, . 
$$
We can use here again Lemma~\ref{fwd:half_pipe:unbdd} and the unique continuation principle to deduce that $\bfu_1^s(\cdot;\bfy_n^*,\bfp) = \bfu_2^s(\cdot;\bfy_n^*,\bfp)$ in $\Omega$, which leads to a contradiction when $n\to\infty$ because $\bfy_n^*\to\bfy^*\in \partial D_1\setminus\overline{D_2}$ (notice $\bfu_1^i(\cdot;\bfy^*,\bfp)$ is singular in $\overline{D_1}$ and $\bfu_2^i(\cdot;\bfy^*,\bfp)$ is not in $\overline{D_2}$). \end{proof}

\subsection{The Near Field Operator and its Basic Properties}\label{wg:inverse:nearfieldop}
%

In the sequel we consider the sources and receivers placed on the same cross-section. 
As suggested by Theorem \ref{inv:uniqueness}, we could allow for different surfaces for the sources and the receivers, but the usual choice for the LSM is to have only one. 

Recall that the {Near Field operator} $N: L_T^2(\Sigma _r)\to L_T^2(\Sigma _r)$ is given by (\ref{inv:nearfieldop}).
%
Note that, in general, given a function $\bfv\in H(\tcl,W_{(-R,R)}\setminus\Sigma_r)$ its tangential trace $\nc\times\bfv$ 
is only in $\widetilde{H}^{-1/2}(\mathrm{div},\Sigma_r)$. Nevertheless, we can define the near field operator from $L_T^2(\Sigma_r)$ into $L_T^2(\Sigma_r)$
, thanks to the following lemma (see \cite[Lemma 3.3.4]{FanPhD}, here restated for a bounded section of the waveguide).
 \begin{lemma}\label{inv:hhdv-to-lt} Given $\bfQ\in\widetilde{H}^{-1/2}(\mathrm{div},\Sigma_{s_1})$, let $\bfU$ denote the solution to the blocked waveguide problem (\ref{fwd:half_pipe:bdd-problem}) posed in $W_{(s_1,R)}$. Then the tangential component $\nc\times\bfU$ 
on any cross-section $\Sigma_{s_2}$ ($s_1<s_2<R$) belongs to $L_T^2(\Sigma_{s_2})$.
\end{lemma}
More precisely, 
 we can rewrite the near field operator by means of two auxiliary operators: On the one hand, we define
$$
\mathcal{H}: \bfg\in L^2_T(\Sigma _r) \mapsto \bfw^i_{\bfg}|_D\in H_{inc}(D) \, ,
$$
where $\, H_{inc}(D)= \{ \bfw\in L^2(D)^3 ; \, \cl\cl\bfw -k^2\bfw = \bfO \,\,\mbox{in } D\} \, $ and
\begin{equation}\label{inv:wgi}
\bfw^i_{\bfg}(\bfx)=\int_{\Sigma_r} \bfu^i(\bfx;\bfy,\bfg(\bfy))\, dS_{\bfy} \quad\mbox{for a.e. }\bfx\in W_{(-R,R)}\setminus\Sigma_r \,
\end{equation}
is 
the \emph{electric single layer potential} on $\Sigma_r$ with density $\bfg$; this operator is linear and bounded.
On the other hand, we define the \emph{incident-to-measurement} operator
$$
\mathcal{N}: \bfw^i\in H_{inc}(D)\mapsto \nc\times\bfw^s|_{\Sigma_r}\in L^2_T(\Sigma_r) \,
$$
where $\bfw^s$ is the solution of the solution of the forward problem (\ref{fwd-sc:problem_truncated}) for the incident field $ \bfw^i$; notice that, by Proposition \ref{fwd:vari_well-posedness} and Lemma \ref{inv:hhdv-to-lt}, we know that this operator is well-defined, linear and bounded.
Then we have the factorization $N=\mathcal{N}\mathcal{H}$. Indeed, we can understand $\bfw^i_{\bfg}=\mathcal{H}\bfg$ as the superposition of the incident fields due to point sources $\bfy\in\Sigma_r$ with polarizations $\bfg (\bfy)$; by linearity of the forward problem, the corresponding scattered field is $$\bfw^s_{\bfg} (\bfx)=\int_{\Sigma_r} \bfu^s(\bfx;\bfy,\bfg(\bfy))\, dS_{\bfy} \quad\mbox{for a.e. } \bfx\in W_{(-R,R)}$$ and, in particular, its tangential component  on $\Sigma_r$ is just $\nc\times\bfw_{\bfg}^s|_{\Sigma_r}= N\bfg$.

In order to analyze the properties of the near field operator, we recall  the following standard homogeneous \emph{Interior Transmission Problem} (ITP):
\begin{equation}\label{inv:itp}
\begin{array}{rcl}
   \cl\cl\bfU_1 - k^2\bfU_1 = \bfO & \mbox{in} & D\, , \\[1ex]
   \cl\cl\bfU_2 - k^2\varepsilon\bfU_2 = \bfO & \mbox{in} & D\, , \\[1ex]
   \nd\times\bfU_1 = \nd\times\bfU_2 & \mbox{on} & \partial D \, , \\[1ex]
   \nd\times(\cl\bfU_1) = \nd\times(\cl\bfU_2) & \mbox{on} & \partial D \, .
\end{array}
\end{equation}
The values of $k$ for which this problem has a nontrivial solution are known as \emph{interior transmission eigenvalues}. Notice that this interior transmission problem is the same that arises in the analysis of the inverse problem in free-space, and has been analyzed in \cite[Section 4.2]{CakoniColtonMonk-IP}.  In particular the set of real transmission
eigenvalues is countable.

 \begin{lemma}\label{inv:NFO:1-1} 
If $k$ is not an interior transmission eigenvalue with an eigenfunction of the form of a single layer potential $\bfU_1=\mathcal{H}\bfg$ ($\bfg\in L^2_T(\Sigma_r)$), then the near field operator $N: L_T^2(\Sigma _r) \to L_T^2(\Sigma _r)$ is one-to-one.
\end{lemma}
\begin{proof}
Let us consider some $\bfg\in L_T^2(\Sigma _r) $ such that $N\bfg =\bfO$ on $\Sigma_r$. By the definition of $N$, this means that 
$$
\int_{\Sigma_r}\nc(\bfx)\times\bfu^s(\bfx;\bfy,\bfg(\bfy))\, dS_{\bfy}=\bfO\quad \mbox{for a.e. } \bfx\in\Sigma _r \, .
$$
In terms of $\bfw_{\bfg}^i=\mathcal{H}\bfg$ 
 and its associated scattered field $\bfw_{\bfg}^s$, the property above means that $\nc\times\bfw^s_{\bfg}=\bfO$ on $\Sigma_r$. Then Lemma \ref{fwd:half_pipe:unbdd} (rewritten for 
 {lower} section of the waveguide $W_{(-R,r)}$, that is, with the radiation condition on $\Sigma_{-R}$ and the boundary condition on $\Sigma_r$) guarantees that $\bfw^s_{\bfg}=\bfO$ in $W_{(-R,r)}$ and, by the unique continuation principle, also in $W_R\setminus \overline{D}$. Therefore, $\bfU_1=\bfw^i_{\bfg}|_D$ and $\bfU_2=(\bfw^i_{\bfg}+\bfw^s_{\bfg})|_D$ solves the ITP with $\bfU_1=\mathcal{H}\bfg$. {By the assumption that $k$ is not a transmission eigenvalue, we have  $\bfU_1=\bfU_2=0$ and hence $\bfg=0$.} This completes the proof. \end{proof}

\begin{lemma}\label{inv:NFO:denserange} 
Under the hypothesis of Lemma \ref{inv:NFO:1-1} on the wavenumber $k$, the range of the near field operator $N: L_T^2(\Sigma _r) \to L_T^2(\Sigma _r)$ is dense in $L_T^2(\Sigma _r)$.
\end{lemma}
\begin{proof}
Equivalently, we study the injectivity of the $L^2$-adjoint of the near field operator, which we denote $N^*: L_T^2(\Sigma _r) \to L_T^2(\Sigma _r)$. To this end,  let us consider $\bfg\in L_T^2(\Sigma _r) $ such that
, for all $\bfh\in L_T^2(\Sigma _r)$,
$$
0=\langle N^*\bfg,\bfh\rangle_{\Sigma _r} = \langle \bfg,N\bfh\rangle _{\Sigma _r} = \int_{\Sigma_r} \bfg(\bfx) \cdot  (\nc(\bfx)\times\int_{\Sigma _r} \overline{\bfu^s(\bfx;\bfy,\bfh(\bfy))}\, dS_{\bfy}  ) dS_{\bfx}\, .
$$
Changing (formally) the order of integration, 
$$
0= \int_{\Sigma_r}  ( \int_{\Sigma _r}  \bfg(\bfx) \cdot  (\nc(\bfx)\times\overline{\bfu^s(\bfx;\bfy,\bfh(\bfy))}  ) dS_{\bfx}  ) dS_{\bfy}\, .
$$
Notice that, making use of the definition of $L^2_T(\Sigma_r)$ and the reciprocity relation,
$$
 \bfg(\bfx) \cdot  (\nc(\bfx)\times\overline{\bfu^s(\bfx;\bfy,\bfh(\bfy))} ) =  - \bfg(\bfx) \cdot \overline{\bfu^s(\bfx;\bfy,\bfh(\bfy))} = 
 - \overline{\bfh(\bfy)} \cdot \overline{\bfu^s(\bfy;\bfx,\overline{\bfg(\bfx)})} \, ,
$$
so that
$$
0= \int_{\Sigma_r} \overline{\bfh(\bfy)} \cdot   ( \int_{\Sigma _r} \overline{\bfu^s(\bfy;\bfx,\overline{\bfg(\bfx)})} dS_{\bfx}  ) dS_{\bfy}\, .
$$
Since this holds for  all $\bfh\in L_T^2(\Sigma _r)$, we deduce that 
$$
\nc ( \bfy) \times \int_{\Sigma _r} \bfu^s(\bfy;\bfx,\overline{\bfg (\bfx)}) \, dS_{\bfx } = \bfO \quad\mbox{for a.e. } \bfy 
\in \Sigma _r \, .
$$
Reasoning as for Lemma \ref{inv:NFO:1-1}, we conclude that $\bfg=\bfO$ on $\Sigma _r$.\end{proof}


In order to study the compactness of the near field operator, we consider the following volume integral representation of the scattered field associated to any incident field $\bfw^i\in H_{inc}(D)$:
\begin{equation}\label{ip:volintrep}
\bfw^s (\bfx) = -\int_D (1-\varepsilon (\bfy)) \bbG_e (\bfx,\bfy) (\bfw^i(\bfy)+\bfw^s(\bfy)) \, d\bfy \quad\mbox{for }\bfx\in W\, .
\end{equation}
Notice that Lemma \ref{inv:green-decompose} guarantees that $\bbG_e (\bfx,\bfy)$ is smooth for $\bfx\in\Sigma _r$ and $\bfy\in D$, from which it follows 
%
that $\mathcal{N}:H_{inc}(D)\to L^2_T(\Sigma_r)$ is compact; thus, the following result is straightforward by the factorization $N=\mathcal{N}\mathcal{H}:L^2_T(\Sigma _r)\to L^2_T(\Sigma_r)$.

 \begin{lemma}\label{inv:NFO:compact} 
The near field operator $N: L_T^2(\Sigma _r) \to L_T^2(\Sigma _r)$ is compact.
\end{lemma}

\subsection{Justification of the Linear Sampling Method}\label{wg:inverse:lsm}

We now give details {of} the  Linear Sampling Method (LSM) outlined in the introduction. We assume that the wavenumber $k$ is not a transmission eigenvalue or an exceptional frequency for the forward problem. More precisely, let us consider a sampling point $\bfy$ in the section of the waveguide $ W_{(r,R)}$, away from its boundary $\Gamma_{(r,R)}$ and in the vicinity of $D$. In order to study if this point $\bfy$ belongs to $D$, we 
fix an artificial polarization $\bfp\in\mathbb{R}^3\setminus\{\boldsymbol{0}\}$ 
 and seek a function $\bfg_{\bfy}\in L^2_T(\Sigma_r)$ that solves  the so-called Near Field equation (\ref{inv:lsm:nfe}).
We then claim that  $\bfy\in D$ if $\Vert\bfg_{\bfy}\Vert_{L^2_T(\Sigma _r)}$ is small.
To provide a justification of this approach, we first characterize the points in $D$ in terms of the range of $\mathcal{N}$.
 \begin{lemma}\label{inv:lsm:inout} 
The tangential trace $\nc\times\bbG_e(\cdot,\bfy)\bfp|_{\Sigma_r}$ is in the range of $\mathcal{N}$ if, and only if, $\bfy\in D$.
Furthermore, when $\{\bfy_n\}_n\subset D$ approaches to a point on the boundary of $D$, 
  the unique solutions of the associated near field equations $\mathcal{N}\bfU^i_{\bfy_n}=\nc\times \bfu^i(\cdot;\bfy_n,\bfp)
|_{\Sigma_r}$
blow up  in the  $H(\mathbf{curl},D)$-norm as $n\to\infty$. 
\end{lemma}
\begin{proof}
Let us consider $\bfy\in D$.  Assuming that $k$ is not an interior transmission eigenvalue, there exists a unique solution of the following non-homogeneous ITP:
\begin{equation}\label{inv:itpaux}
\begin{array}{rcl}
   \cl\cl\bfU_{\bfy}^1 - k^2\bfU^1_{\bfy} = \bfO & \mbox{in} & D\, , \\[1ex]
   \cl\cl\bfU_{\bfy}^2 - k^2\varepsilon\bfU_{\bfy}^2 = \bfO & \mbox{in} & D\, , \\[1ex]
   \nd\times(\bfU^2_{\bfy}-\bfU^1_{\bfy}) = \nd\times\bfu^i(\cdot;\bfy,\bfp) & \mbox{on} & \partial D \, , \\[1ex]
   \nd\times (\cl (\bfU^2_{\bfy}-\bfU^1_{\bfy}) )= \nd\times(\cl \bfu^i(\cdot;\bfy,\bfp)) & \mbox{on} & \partial D \, .
\end{array}
\end{equation}
Notice that
\begin{equation}\label{nfe:defUsy}
\bfU ^s_{\bfy} = \left\{
\begin{array}{ll}
\bfu^i(\cdot;\bfy,\bfp) & \mbox{in } W_R\setminus\overline{D}\, ,\\[1ex]
\bfU^2_{\bfy}-\bfU^1_{\bfy} \quad & \mbox{in } {D}\, ,
\end{array} \right.
\end{equation}
solves the forward problem (\ref{fwd-sc:problem_truncated}) for the incident field $\bfU_{\bfy}^1\in H_{inc}(D)$. Therefore
$$
\mathcal{N} \bfU_{\bfy}^1 = \nc\times \bfU^s_{\bfy}|_{\Sigma _r} = \nc\times \bfu^i(\cdot;\bfy,\bfp)|_{\Sigma _r} \, ,
$$
and we deduce that $\boldsymbol{\nu}_0\times \bfu^i(\cdot;\bfy,\bfp|_{\Sigma _r})\in \mathcal{N}  (H_{inc}(D))$.

Reciprocally, let us take $\bfy\in W\backslash\overline{D}$ and suppose that there exists $\bfU_{\bfy}^0\in H_{inc}(D) $ such that $\mathcal{N}\bfU_{\bfy}^0 = \nc\times\bfu^i(\cdot;\bfy,\bfp)
|_{\Sigma_r}$. If $\bfU^s_{\bfy}\in {X}_{(-R,R)}$ denotes the associated scattered wave
, then 
$$
 \nc\times\bfU^s_{\bfy}|_{\Sigma _r}  =
\mathcal{N}\bfU_{\bfy}^0 = 
 \nc\times\bfu^i(\cdot;\bfy,\bfp)
|_{\Sigma_r}
 \, .
$$
Therefore, Corollary \ref{fwd:half_pipe:bdd} rewritten on the section $W_{(-R,r)}$ 
guarantees that $\bfu^i(\cdot;\bfy,\bfp)
= \bfU^s_{\bfy}$ in $W_{(-R,r)}$ and, by the unique continuation principle, also in $W_{R}\setminus(\overline{D}\cup \{\bfy\})$. This leads to a contradiction when approaching  $\bfy\in W_{(-R,R)}\setminus\overline{D}$ {as $\bfu^i$ is singular there} and hence $\boldsymbol{\nu}_0\times \bfu^i(\cdot;\bfy,\bfp|_{\Sigma _r})\not\in \mathcal{N}  (H_{inc}(D))$. 

We next study the behavior of the solutions of
$$\mathcal{N}\bfU^i_{\bfy_n}=\nc\times\bfu^i(\cdot;\bfy_n,\bfp) 
|_{\Sigma_r}$$
for a sequence $\{\bfy_n\}_n\subset D$ that approaches to a point on the boundary of $D$. To this end, we recall that we have built these solutions using the incident and scattered fields $\bfU^i_{\bfy_n}=\bfU_{\bfy_n}^1$ and  $\bfU ^s_{\bfy_n} $ which satisfy and  (\ref{inv:itpaux}) and (\ref{nfe:defUsy}), respectively; then
$$
\mathcal{N} \bfU_{\bfy_n}^1 = \nc\times  \bfU^s_{\bfy_n}|_{\Sigma _r} = \nc\times \bfu^i(\cdot;\bfy_n,\bfp)|_{\Sigma _r} \, .
$$
To study the behavior of these functions, let us fix two auxiliary Lipschitz domains ${D}_1$ and ${D}_2$ such that
$$
\overline{D}\subset D_1 \, , \quad \overline{D_1}\subset D_2 \, , \quad \overline{D_2}\subset W_{(-R,R)} \, .
$$
Notice that 
\begin{equation*}
\begin{array}{rcl}
   \cl\cl\bfU^s_{\bfy_n} - k^2\varepsilon\bfU^s_{\bfy_n} = - k^2 (1-\varepsilon) \bfU_{\bfy_n}^1 & \mbox{in} & D_2 \, , \\[1ex]
   \boldsymbol{\nu}_{D_2} \times \bfU^s_{\bfy_n} = \boldsymbol{\nu}_{D_2}\times\bfu^i(\cdot;\bfy_n,\bfp) & \mbox{on} & \partial D_2 \, , \\[1ex]
   \boldsymbol{\nu}_{D_2} \times (\cl \bfU^s_{\bfy_n} )= \boldsymbol{\nu}_{D_2} \times(\cl \bfu^i(\cdot;\bfy_n,\bfp)) & \mbox{on} & \partial D_2 \, ,
\end{array}
\end{equation*}
where $\boldsymbol{\nu}_{D_2} $ stands for the unit normal field on $\partial D_2$ directed outwards. The above is a well-posed forward problem, in particular
$$
\begin{array}{l}
\V \bfU^s_{\bfy_n} \V _{H(\mathbf{curl},D_2)}\, \leq \, C\, \Big( 
 \V \bfU^1_{\bfy_n} \V _{H(\mathbf{curl},D)} 
+ \V \boldsymbol{\nu}_{D_2}\times\bfu^i(\cdot;\bfy_n,\bfp) \V _{H^{-1/2}_T (\partial D_2)} \Big. \\
\hspace*{6cm} \Big. + \V \boldsymbol{\nu}_{D_2}\times (\cl\bfu^i(\cdot;\bfy_n,\bfp)) \V _{H^{-1/2}_T (\partial D_2)}
\Big) ;
\end{array}
$$
hence, by the continuity of the tangential trace and that 
$\cl\cl\bfu^i(\cdot;\bfy_n,\bfp)=k^2\bfu^i(\cdot;\bfy_n,\bfp)$ in $D_2\setminus\overline{D_1}$, 
$$
\V \bfU^s_{\bfy_n} \V _{H(\mathbf{curl},D_2)} \leq C  \left( 
 \V \bfU^1_{\bfy_n} \V _{H(\mathbf{curl},D)} 
+ \V \bfu^i(\cdot;\bfy_n,\bfp) \V _{H(\mathbf{curl}, D_2\setminus\overline{D_1})}
\right) .
$$
Finally notice that, when $\{\bfy_n\}_n\subset D$ approaches  a point on the boundary of $D$, the sequence $\V \bfu^i(\cdot;\bfy_n,\bfp) \V _{H(\mathbf{curl}, D_2\setminus\overline{D_1})}$ remains bounded whereas $\V \bfU^s_{\bfy_n} \V _{H(\mathbf{curl},D_2)} \geq \V \bfu^i(\cdot;\bfy_n,\bfp) \V _{H(\mathbf{curl},D_2\setminus\overline{D})} $ blows up; therefore, we conclude that $ \V \bfU^1_{\bfy_n} \V _{H(\mathbf{curl},D)} $ must also blow up.
\end{proof}

We continue our analysis of the LSM with a result which justifies the usage of single layer potentials in the near field equation. {In this proof, we use the following explicit expression of the electric Green's function in terms of the waveguide modes introduced in Section \ref{subsec:forward:modal}, cf. \cite[Subsection 3.3.1]{FanPhD} and \cite{ChenToTai}:
        \begin{equation}\label{ip:GeModes}
          \bbG_{e}(\bfy,\bfx) \, = \, \displaystyle\sum_{m=1}^{\infty} c_m \bfM_m(\bfy^-){\bfM_m(\bfx)}^T 
					+ \sum_{n=1}^{\infty} d_n \bfN_n(\bfy^-){\bfN_n(\bfx)}^T \quad\mbox{for } y_3<x_3\, .
				\end{equation}
In this expression, the superindex $^-$ represents the reflection of a point with respect to the plane $z_3=0$, that is, $\bfy^-:=(\hat{y},-y_3)$ when $\bfy=(\hat{\bfy},y_3)$. Moreover, the coefficients $c_m,d_n\ (m,n=1,2,\ldots)$ depend on 
the shape of $\Sigma$, and the terms $\bfM_m(\bfy^-)\,{\bfM_m(\bfx)}^T$ and $\bfN_n(\bfy^-)\,{\bfN_n(\bfx)}^T$ denote the $3\times 3$ matrices obtained by column-row multiplication.}

 \begin{lemma}\label{inv:densHinc}
Let us assume that 
{the forward problem for the waveguide with a PEC boundary condition on $\partial D$ is well-posed in $W\setminus\overline{D}$ and for the interior Maxwell problem on $D$ (for $\varepsilon = 1$)}.  {We also assume that all the coefficients in the expansion (\ref{ip:GeModes}) are non zero.} Then, the operator $\mathcal{H}: L^2_T(\Sigma _r) \to H_{inc}(D)$ has dense range.
\end{lemma}
\begin{remark}{Using the results in \cite{ChenToTai} pages 108 and  141, we know that the assumption on the
coefficients in the expansion (\ref{ip:GeModes}) is satisfied for a rectangular or circular cross-section waveguides.}
\end{remark}
\begin{proof}
We show that the 
adjoint of $\mathcal{H}: L^2_T(\Sigma _r) \to H_{inc}(D)$ is one-to-one. To this end, we consider $\bfU\in H_{inc}(D)'$ such that $\mathcal{H}^*\bfU=\bfO$ in $L^2_T(\Sigma_r)$, that is, for all  $\bfg\in L^2_T(\Sigma_r) $ it holds
$$
0= \langle \mathcal{H}^*\bfU,\bfg \rangle_{\Sigma _r} = (\bfU,\mathcal{H}\bfg )_{D} 
=
\int_D \bfU (\bfx) \cdot  ( \int_{\Sigma_r} \overline{ \mathbb{G}_e (\bfx,\bfy) \bfg(\bfy)} \, dS_{\bfy}  ) d\bfx
\, ;
$$
using that $\bfU (\bfx) \cdot  ( \overline{ \mathbb{G}_e (\bfx,\bfy) } \overline{\bfg(\bfy)}  ) =  ( \overline{ \mathbb{G}_e (\bfx,\bfy) }^T  \bfU (\bfx)  ) \cdot \overline{\bfg(\bfy)} $ and changing (formally) the order of integration,
$$
\int_{\Sigma_r}  
 ( \int_D \overline{ \mathbb{G}_e (\bfx,\bfy) }^T \bfU (\bfx) \,  d\bfx  ) 
\cdot \overline{\bfg(\bfy)}\, dS_{\bfy} 
\, =\, 0 \, ;
$$
since this holds for all $\bfg\in L^2_T(\Sigma_r)$, what we have is that $\nc\times\bfV_{\bfU}|_{\Sigma_r}=\bfO$ on $\Sigma_r$, 
where 
$$
\bfV_{\bfU} (\bfy ) = \int_D \mathbb{G}_e (\bfx,\bfy) \overline{ \bfU (\bfx) } \,  d\bfx \quad\mbox{for }\bfy\in W\setminus\overline{D}\, . 
$$


{Substituting (\ref{ip:GeModes}) in the definition of $\bfV_{\bfU} $,} we have
$$
\bfV_{\bfU} (\bfy ) =
          \displaystyle\sum_{m=1}^{\infty}\! c_m \,\bfM_m(\bfy^-) \!\int_D\! {\bfM_m(\bfx)}^T  \overline{ \bfU (\bfx) } \,  d\bfx
					\, +\, \sum_{n=1}^{\infty} \! d_n\, \bfN_n(\bfy^-) \!\int_D \!{\bfN_n(\bfx)}^T  \overline{ \bfU (\bfx) } \,  d\bfx\, .
$$
This allows us to rewrite the boundary condition $\nc\times\bfV_{\bfU}|_{\Sigma_r}=\bfO$ on $\Sigma_r$  as 
				\begin{equation}\label{ip:Ucoefnull}
          \int_D \bfM_m(\bfx)^T \overline{\bfU (\bfx)} \, d\bfx  = 0 \, , 
					\quad 
					\int_D \bfN_n(\bfx)^T \overline{\bfU (\bfx) } \, d\bfx = 0 
					\qquad\forall m,n=1,2,\ldots
				\end{equation}
thanks to the definition of $\bfM_m, \bfN_n$ (see Section \ref{subsec:forward:modal}) and that $\{ \sgd u_m\}_{m=1}^{\infty} \cup \{ \vgd v_n\}_{n=1}^{\infty}$ defines an orthonormal basis of $ L^2_T(\Sigma _r)$ (see \cite[Lemma 3.1.2]{FanPhD}); notice that this reasoning also requires the  further assumption $c_m\neq 0$ and $d_n\neq 0$ for all $m,n=1,2,...$ 

In order to finally conclude that $\bfU$ {vanishes, note} that the operator
$$
\mathcal{S}_{\partial D}: \bfh \in H^{-1/2}(\mathbf{curl},\partial D) \mapsto \mathcal{S}_{\partial D} \bfh=\boldsymbol{\nu}_D \times\int_{\partial D} \bbG_e(\cdot,\bfz) \, \bfh(\bfz) \, d\bfz \in H^{-1/2}(\mathrm{div}, \partial D) \, ,
$$
defines an isomorphism, cf. \cite[Lemma 3.3.5]{FanPhD}.  In consequence, by the well-posedness of the interior problem that characterizes the space $H_{inc}(D)$, we know that the linear operator
$$
\mathcal{H}_{\partial D}: \bfh \in H^{-1/2}(\mathbf{curl},\partial D) \mapsto \mathcal{H}_{\partial D} \bfh=\int_{\partial D} \bbG_e(\cdot,\bfz) \, \bfh(\bfz) \, d\bfz \in H_{inc}(D) \,
$$
has dense range; and, in particular, from the expression of the fundamental solution in terms of modes (\ref{ip:GeModes}) it follows that
$$
\mathrm{span}\{ \bfM_m,\bfN_n ; \, m,n=1,2,\ldots \}
$$
is dense in $H_{inc} (D)$. Therefore, (\ref{ip:Ucoefnull}) already guarantees that $\bfU\in H_{inc}(D)'$ cancels.
 \end{proof}

Now, we shall prove the main result for the justification of the LSM under the additional assumption of Lemma \ref{inv:densHinc} on the wavenumber $k$.

\begin{theorem}\label{inv:lsm:justification} 
%
%
Let us fix any polarization $\bfp\in\mathbb{R}^3\setminus\{\bfO\}$. \begin{enumerate}
\item For each $\bfy\in D$ and $\epsilon>0$, there exists $\bfg_{\bfy}^{\epsilon}\in L_T^2(\Sigma_r)$ such that
        \begin{eqnarray}
            \V N\bfg_{\bfy}^{\epsilon} - \nc\times\bfu^i (\cdot;\bfy,\bfp)
\V_{L_T^2(\Sigma_r)} < \epsilon \, , \label{inv:lsm:indicator-norm}
        \end{eqnarray}
        and for which 
the associated scattered fields
				$$\bfw^s_{\bfg_{\bfy}^{\epsilon}}(\bfz) = \int_{\Sigma_r}\bfu^s(\bfz;\bfx,\bfg_{\bfy}^{\epsilon}(\bfx))\,dS_{\bfx}$$
converge to 
 $\bfu^s(\cdot ;\bfy,\bfp)$ 
in $H(\tcl,W_R)$ as $\epsilon arrow 0$; moreover, if a sequence $\{\bfy_n \}_n\subset D$ approaches to some point on $\partial D$, then necessarily these functions blow up:
        $$ \lim_{n}\V\bfg_{\bfy_n}^{\epsilon}\V_{L_T^2(\Sigma_r)} = \infty\, . $$
%
\item If $\bfy\not\in D$, any sequence $\{\bfg_{\bfy}^{\epsilon}\}_{\epsilon>0}\subset L_T^2(\Sigma_r)$ that satisfies (\ref{inv:lsm:indicator-norm}) must also blow up:
        $$ \lim_{\epsilon \to 0}\V\bfg_{\bfy}^{\epsilon}\V_{L_T^2(\Sigma_r)} = \infty\, . $$
\end{enumerate}
\end{theorem}
\begin{proof}
Thanks to the factorization $N=\mathcal{N}\mathcal{H}$, the first statement follows from  the characterization of points $\bfy\in D$ in terms of $ \nc\times\bfu^i(\cdot ;\bfy,\bfp)
\in \mathcal{N}(H_{inc}(D)) $ given in Lemma \ref{inv:lsm:inout} and the density of $\mathcal{H}(L_T^2(\Sigma_r))$ in $H_{inc}(D)$ shown in Lemma \ref{inv:densHinc}.

On the other hand, let us consider a point $\bfy\in W_R\setminus\overline{D}$ and  a bounded sequence $\{\bfg_{\bfy}^{\epsilon}\}_{\epsilon>0}\subset L_T^2(\Sigma_r)$ satisfying (\ref{inv:lsm:indicator-norm}). Then there exists a subsequence of $\{\bfg_{\bfy}^{\epsilon}\}_{\epsilon>0}$ that converges weakly to some $\bfg_{\bfy}$  in $L_T^2(\Sigma_r)$; we arrive to a contradiction by noticing that $ \nc\times\bfu^i(\cdot;\bfy,\bfp)
=N\bfg =\mathcal{NH}\bfg\in\mathcal{N}(H_{inc}(D))$ for the point $\bfy\not\in D$.
\end{proof}


\subsection{Some remarks on the Generalized Linear Sampling Method}\label{glsm}

In this subsection we provide some remarks about the Generalized Linear Sampling Method (GLSM) for the inverse problem under study. More precisely, let us recall that the LSM makes use of an approximate solution of the NFE 
in the sense of (\ref{inv:lsm:indicator-norm}); usually, this is done by means of a Tikhonov regularization so that the following is minimized:
\begin{equation*}
            \V N\bfg_{\bfy}^{\alpha} - \nc\times\bfu^i (\cdot;\bfy,\bfp)
\V_{L_T^2(\Sigma_r)}^2 +  \alpha^2 \,\V \bfg_{\bfy}^{\alpha} \V_{L_T^2( \Sigma _r)} \, ,
\end{equation*}
for $\bfg_{\bfy}^{\alpha}\in L^2_T(\Sigma _r)$. 
{In contrast,} the noise free 
GLSM that we consider here approximately solves the NFE 
 by minimizing
\begin{equation}\label{inv:glsm:nfe}
            \V N\bfg_{\bfy}^{\alpha} - \nc\times\bfu^i (\cdot;\bfy,\bfp)
\V_{L_T^2(\Sigma_r)}^2 +  \alpha^2\, |\langle N \bfg_{\bfy}^{\alpha} ,  \bfg_{\bfy}^{\alpha} \rangle_{L_T^2( \Sigma _r)} | \, .
\end{equation}

To analyze this strategy, we first notice that, for all $\bfv\in L^2(D)^3$ and $\mathbf{f}\in L^2_T(\Sigma_r)$
\begin{equation*}
(\mathcal{H}^*\mathbf{v},\bff)_{\Sigma _r} = \int _D \mathbf{v}\cdot \overline{\mathcal{H}\bff} \, d\bfx 
= \int_{\Sigma _r} \overline{ \bff (\bfy) } \cdot  ( \nc \times\int_D \overline{\bbG_e (\bfy,\bfx)} \bfv (\bfx) \,d\bfx  ) dS_{\bfy} \, ,
\end{equation*}
where we have changed the order of integration and made use of the symmetry of the dyadic function $\bbG_e(\cdot,\cdot)$. The above means that, for all $\bfv\in L^2(D)^3$,
\begin{equation*}
{\mathcal{H}^*\mathbf{v}} = \nc \times \int_D \overline{\bbG_e (\cdot,\bfx) } \bfv (\bfx) \, d\bfx \quad \mbox{on } \Sigma _r \, ;
\end{equation*}
therefore, taking into account the volume integral representation (\ref{ip:volintrep}) and the  definition of the auxiliary operator $\mathcal{N}: \bfw^i\in H_{inc}(D)\mapsto \nc\times\bfw^s|_{\Sigma_r}\in L^2_T(\Sigma_r)$, we deduce that
\begin{equation*}
-\overline{\mathcal{H}^*(k^2\, \overline{(1-\varepsilon)\mathbf{w}})} = \nc \times \bfw^s = \mathcal{N} \bfw^i \quad \mbox{on } \Sigma _r \, ,
\end{equation*}
for each $\bfw^i\in H_{inc}(D)$, where $\bfw^s$ and $\bfw=\bfw^i+\bfw^s$  denote the corresponding scattered and total fields. In particular, using the factorization of the near field operator $N=\mathcal{N}\mathcal{H}$, we have that 
\begin{equation}\label{eq:Nfactization}
N\bfg \, = \, \mathcal{N}\bfw^i_{\bfg} = 
\overline{\mathcal{H}^*(\overline{\mathcal{T}\bfw^i_{\bfg}})} =
\overline{\mathcal{H}^*(\overline{\mathcal{T}\mathcal{H}\bfg})} 
\quad\mbox{for all } \bfg\in L^2_T(\Sigma_r) \, ,
\end{equation}
where we have used a third auxiliary operator, defined by
\begin{equation*}
\mathcal{T}: \mathbf{w}^i \in H_{inc}(D) \to -k^2 \, (1-\varepsilon ) \, \bfw|_D 
\in L^2(D)^3 \, .
\end{equation*}
The key property for the analysis of the GLSM is the coercivity of this operator, see \cite{ComputELectromag-Cetraro}.
In order to study such  coercivity, we notice that, for all $\bfw^i\in H_{inc}(D)$,
\begin{equation}\label{ip:productTwiwi}
(\mathcal{T}\bfw^i,\bfw^i)_D = -k^2 \int_D (1-\varepsilon) \, |\bfw^i|^2 \, d\bfx -k^2 \int_D (1-\varepsilon) \, \bfw^s\cdot\overline{\bfw^i} \, d\bfx \, ;
\end{equation}
the second term on the right-hand side can be analyzed using the equation (\ref{fwd-total:weak_form}) that characterizes the scattered field $\bfE=\bfw^s$ for the test function $\bfv=\bfw^i$:
\begin{equation*}
\begin{array}{l}
\displaystyle\-k^2 \int_D (1-\varepsilon) \, \bfw^s\cdot\overline{\bfw^i} \, d\bfx 
= 
\int_{W_{(-R,R)}}\!\! ( |\cl\bfw^s|^2-k^2\,\overline{\varepsilon} \, |\bfw^s|^2  ) d\bfx
+ \displaystyle\int_{\Sigma_{\pm R}} \!\!\overline{T^{ \pm }_{R} (\nc \times \bfw^s) } \cdot \boldsymbol{\gamma}_T \bfw^s \, dS_{\bfx} 
\, ;
\end{array}
\end{equation*}
moreover, we can write explicitly the imaginary part of last term using the modal expansion (\ref{fwd-dtn-explicit-form}):
\begin{equation*}
\begin{array}{l}
\Im \big(\!\displaystyle\int_{\Sigma_{\pm R}}\!\!\overline{T^{ \pm }_{R} (\nc \times \bfw^s) } \cdot \boldsymbol{\gamma}_T \bfw^s \, dS_{\bfx}\big)  = 
-\displaystyle\sum_{m=1}^{m_0} |\langle\nc\times\bfw^s ,  \left(\!\!\begin{array}{c} \sgd u_m \\ 0 \end{array}\!\!\right)  \rangle _{\Sigma_{ \pm R}}|^2 \,\frac{h_m}
{\lambda_m^2}\\[1ex]
\hspace*{2cm}-\displaystyle\sum_{n=1}^{n_0} |\langle\nc\times\bfw^s ,  \left(\!\!\begin{array}{c} \vgd v_n \\ 0 \end{array}\!\!\right)  \rangle _{\Sigma_{\pm R}}|^2 \,\frac{k^2}{\mu_n^2\,g_n} 
 \, ,
\end{array}
\end{equation*}
where $m_0$ and $n_0$ represent the indices up to which  
$k^2>\lambda_m^2$ and $k^2>\mu_n^2$, respectively; in other words, the imaginary part of this term catches the asymptotics of the traveling waves. Taking this results back to (\ref{ip:productTwiwi}), we deduce that
\begin{equation*}
\begin{array}{l}
\displaystyle
\Im ((\mathcal{T}\bfw^i,\bfw^i)_D) \, = \, k^2\int_D \Im ( \varepsilon ) \, ( |\bfw^i|^2 +\, |\bfw^s|^2 ) \, d\bfx +
\displaystyle\displaystyle\sum_{m=1}^{m_0} |\langle\nc\times\bfw^s ,  \left(\!\!\begin{array}{c} \sgd u_m \\ 0 \end{array}\!\!\right)  \rangle _{\Sigma_{ \pm R}}|^2 \, \frac{h_m}{\lambda_m^2}
\\
\hspace*{4cm}
+
\displaystyle\sum_{n=1}^{n_0} |\langle\nc\times\bfw^s ,  \left(\!\!\begin{array}{c} \vgd v_n \\ 0 \end{array}\!\!\right)  \rangle _{\Sigma_{ \pm R}}|^2 \, \frac{k^2}{\mu_n^2\, g_n} \, 
\end{array} 
\end{equation*}
from where we conclude that $\mathcal{T}$ is coercive whenever the imaginary part of $\varepsilon$ is strictly positive in some subdomain of $D$ with non-zero measure. We are ready to prove the following justification of the GLSM for our problem.

\begin{theorem}
Let us assume 
that the imaginary part of  $\varepsilon$ 
is strictly positive in $D$ (or on a subdomain  with non-zero measure). Then, for any polarization $\bfp\in\mathbb{R}^3\setminus\{\bfO\}$ and $\bfy\in W$, it holds that $\bfy\in D$ if, and only if, any sequence $\{ \bfg^{\alpha}_{\bfy}\}_{\alpha>0}\subset L^2_T(D)$ of minimizers of 
 (\ref{inv:glsm:nfe}) is bounded.
\end{theorem}
\begin{proof}
On one hand, points $\bfy\in D$ are characterized by the property $ \nc\times\bfu^i(\cdot;\bfy,\bfp)
|_{\Sigma_r}
\in \mathcal{N}(H_{inc}(D)) $ (see Lemma \ref{inv:lsm:inout}). On the other hand, the near field operator can be factorized both as $N=\mathcal{N}\mathcal{H}$ and as in (\ref{eq:Nfactization}); moreover, $\mathcal{N}:H_{inc}(D)\to L^2_T(\Sigma_r)$ has dense range and is compact (see Lemmas \ref{inv:NFO:denserange} and \ref{inv:NFO:compact}, respectively), whereas $\mathcal{T}:H_{inc}(D)\to L^2(D)^3$ is coercive when the imaginary part of $\varepsilon$ is strictly positive (see the reasoning above the statement). Therefore, the result follows by \cite[Chapter 4, Theorem 8]{ComputELectromag-Cetraro}.
\end{proof}

%
\section{Numerical results}\label{Sec-Numresult}

In this section, we shall describe some numerical simulations of the reconstruction of scattering objects in order to investigate the application of the LSM to inverse electromagnetic scattering in a waveguide. Specifically, we use NGSolve~\cite{netgen} to implement a forward scattering code to generate synthetic scattering data to be collected at the receivers located on a cross-section of the waveguide  below the scatterer.  In particular we used quadratic edge finite elements to approximate $\bfE^s$ on a finite section of the waveguide, and terminate this section at both ends using the non-standard Perfectly Matched Layer (PML) proposed in~\cite{Rivas} (with the parameters used there).  This PML is singular and accounts for both traveling and evanescent components of the solution.  The electric field
is extended to the entire waveguide by using a truncated modal expansion (\ref{modal}) to one side of the scatterer.  The expansion coefficients are computed by fitting the finite element solution on a cross section
of the waveguide taken to be an interface in the mesh, and we always use up to order 7 Fourier modes in $x$ and $y$ giving 63 modes for each polarization.  The mesh size suggested to NGSolve is $2\pi/(7 k)$ in the air
and $2\pi/(7 k \sqrt{\epsilon})$ in $D$.

The waveguide is taken to have a square cross-section $\Sigma=(0, 1)\times(0,1)$, and the scatterer $D$ has a constant electric permittivity $\varepsilon=4$. We consider two wavenumbers: $k=20$ and $k=25$.  When $k=20$ we have 38 propagating modes, and when $k=25$ we have 55 such modes (in this case the highest Fourier order for a propagating mode  is 7). For all experiments we use an $8\times 8$ grid of transducers (the same points are used to place the sources and to take measurements) at the tensor product Gauss-Legendre quadrature points in $\Sigma_r$ where we choose $r=-5$.   At each source point we use successively each of the three polarizations parallel to the coordinate axes, and assume knowledge of all three polarizations of the scattered field at the measurement points.  Using the product Gauss-Legendre quadrature scheme to discretize the near field operator results in a $192\times 192$ \emph{near field matrix}.  

In this paper we use a simple spectral cutoff regularization which appears
sufficient for the examples here although the more standard Tikhonov-Morozov scheme~\cite{CakoniColtonMonk-IP} might be preferable in practice.  We choose the spectral cutoff manually as described below. 
In some cases noise is added to the data entry by entry as described in \cite{CakoniColtonMonk-IP}. In particular if $N$ is the matrix representing the near field operator after discretization using Gauss-Legendre quadrature (in our case $N$ is a $192\times 192$ matrix) then, for a given noise parameter $\eta$, we add noise by computing a new matrix $N_{\eta}$ using
\[
(N_{\eta})_{i,j}=N_{1,j}(1+\eta \xi_{i,j})\mbox{ for all }i,j.
\]
Here $\xi_{i,j}$ is drawn from a set of random numbers uniformly distributed in $(-1,1)$ (using the Numpy \verb+random.uniform+ command). 

Having computed a regularized solution of the discrete near field equation for each of three linearly independent auxiliary polarizations (taken to be each of the three standard unit vectors successively) due to a given
source point $\bfz$, we use the reciprocal of the average of the discrete $\ell^2$-norms of the discrete approximation to $\bfg_z$ for each of these three polarizations as the indicator function for identification of the shape of the scatterer.
We shall present isosurface plots for the indicator function as well as detailed contour plots on cross-sections of the 
domain.  The isosurface to draw is chosen by fixing a constant $0<C<1$, and the isosurface  is then given by all
$\mathbf{x}$ such that
\begin{equation}
\psi(\mathbf{x})=C  ( \max _{\mathbf{y}\in Z} \psi(\mathbf{y})-\min_{\mathbf{y}\in Z} \psi(\mathbf{y}) )+\min _{\mathbf{y}\in Z} \psi(\mathbf{y}) \, ,
\label{level}
\end{equation}
where $\psi$ is the indicator function and $Z$ is the set of source points used for the sampling method.
The constant $C$ may have to be modified for different scatterers and noise levels.

\begin{figure}
\begin{center}
\begin{tabular}{cc}
\resizebox{0.3\textwidth}{!}{
\includegraphics{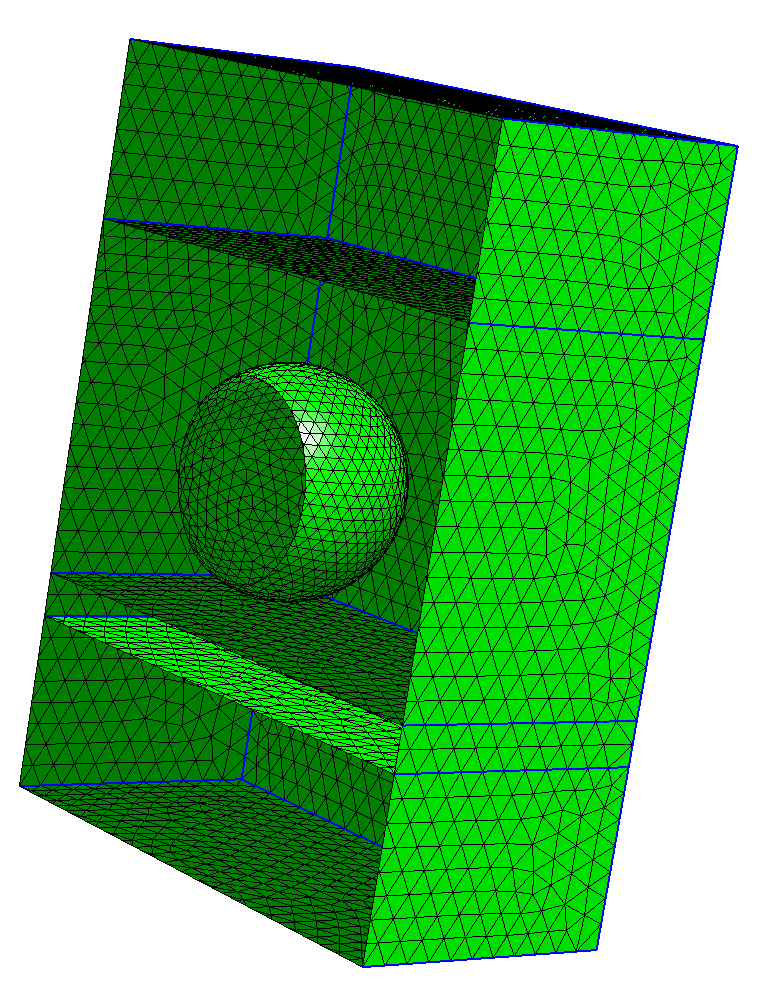}} \hspace{.0125\textwidth} & \hspace{.0125\textwidth} 
\resizebox{0.3\textwidth}{!}{
\includegraphics{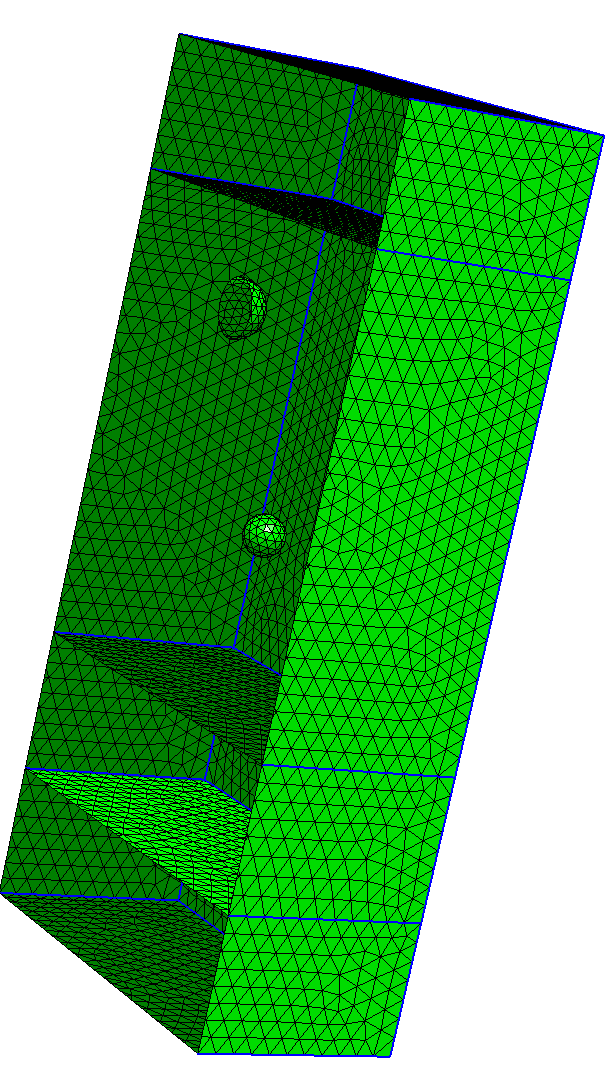}}
\end{tabular}
\end{center}
\caption{Cross-sections of the computational domains used for the forward problems in this study.  Left: the single spherical scatterer. Right: two smaller spheres.  The source points and measurement surface are below the scatterers.  The boxed regions at the top and bottom of the domain are the PML regions (in this case $k=20$ and the PML regions are one wavelength thick).  The surface between the obstacle or obstacles and the lower PML is used to fit a modal expansion for extending the solution outside the computational domain.}
\label{domains}
\end{figure}

We will consider two examples motivated by previous works in the area. The examples are three-dimensional analogues of the examples in \cite{BLFD}  (see Fig.~\ref{domains}).  The scatterer $D$ is chosen to be either:
\begin{itemize}
\item A single sphere centered at $(0.5,0.6,0)$ with radius 0.2.
\item Two spheres, the first centered at $(0.5,0.7,0)$ of radius 0.05, and the second centered at $(0.5,0.5,0.5)$ of radius 0.07.
\end{itemize}
%

\subsection{A single Sphere}
First we consider the single sphere of radius 0.2 centered at $(0.5,0.6,0)$.
Of the two examples considered in this paper, this is the most difficult to reconstruct. As discussed above, regularization
is via the truncated singular value decomposition.  In Fig.~\ref{S20} we show the singular values when $k=20$.  In the left panel no extra noise has been added to the data computed by the finite element method. In the right-hand panel we have added random noise with noise parameter $\eta=0.001$  that produces a relative error in the discrete near field matrix of 0.06\%  in the Frobenius norm. {We also show examples for} $\eta=0.01$ which produces an error of 0.6\%. When $k=20$ there are 38 propagating modes.  We choose a spectral cutoff larger than this, restricting to the
first 51 SVD vectors.  With this choice the reconstructions are shown in Fig.~\ref{onesphk20}.  The position of the scatterer along the waveguide is predicted well, but the shape is not obvious from the isovalue plots even for no added noise.

A higher wavenumber results in more propagating modes, and hence possibly more data.  Using $k=25$, when there are 55 propagating modes, 
    gives the singular values shown in Fig.~\ref{S25}.  In this case more singular vectors are significant (compared to the $k=20$ case).  In the case of no noise shown in Fig.~\ref{onesphk25} (left panel) we used 81 singular vectors.  The position and approximate shape of the scatterer are clearly visible. We compare in Fig.~\ref{onesphk25} center and right panels the effect of noise.  {To compute the results shown in Fig.~\ref{onesphk25} (center panel),} we use 81 singular values results in an improved reconstruction compared to Fig.~\ref{onesphk20}.
 \begin{figure}
\begin{center}
\begin{tabular}{cc}
\resizebox{0.45\textwidth}{!}{
\includegraphics{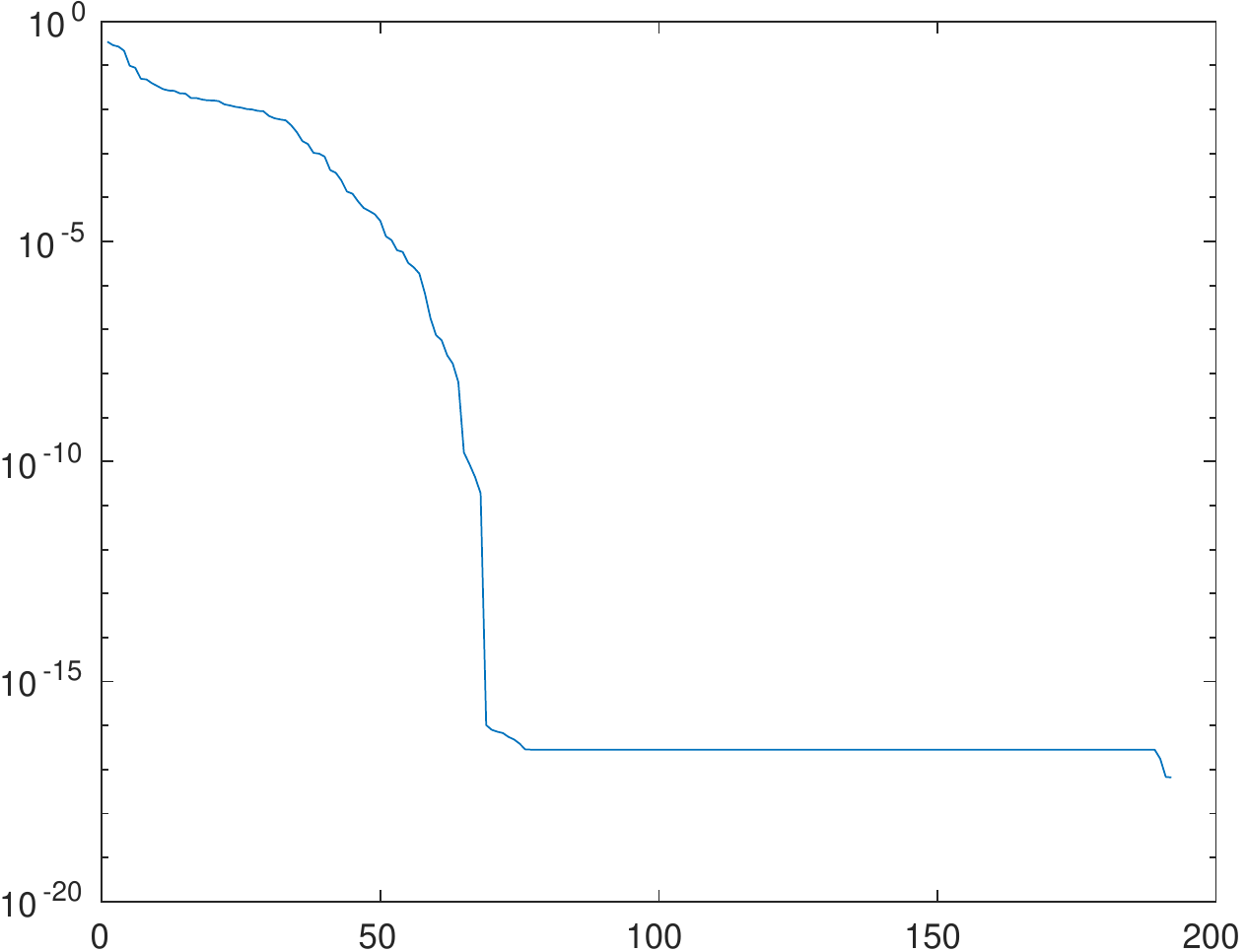}}&
\resizebox{0.45\textwidth}{!}{
\includegraphics{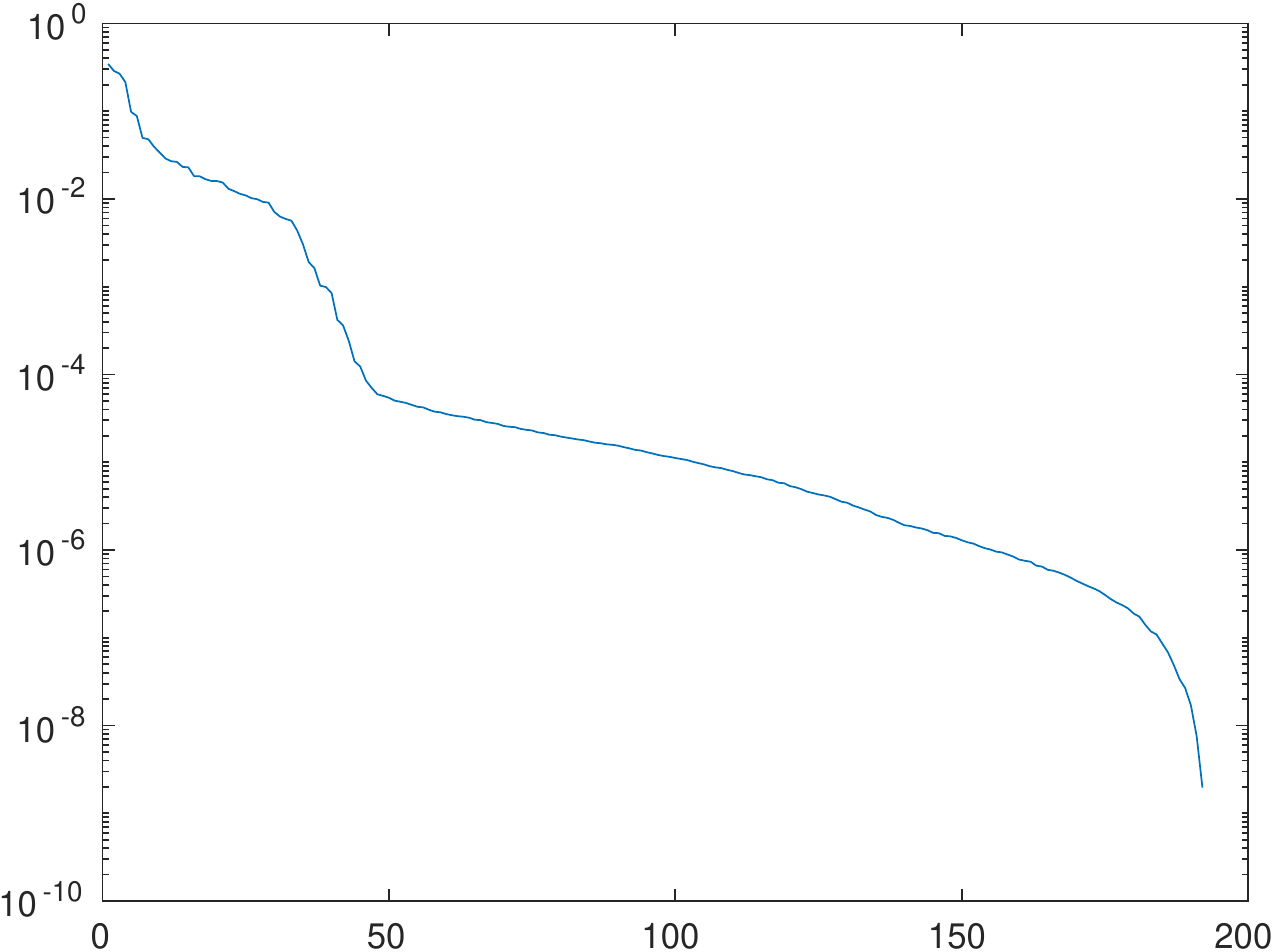}}
\end{tabular}
\end{center}
\caption{Singular values for the single sphere example shown in Fig.~\ref{domains}. The wavenumber is $k=20$.
Left: no noise. 
Right:  with added noise $\eta=0.001$.}
\label{S20}
\end{figure}

\begin{figure}
\begin{center}
\begin{tabular}{ccc}
\resizebox{0.3\textwidth}{!}{
\includegraphics{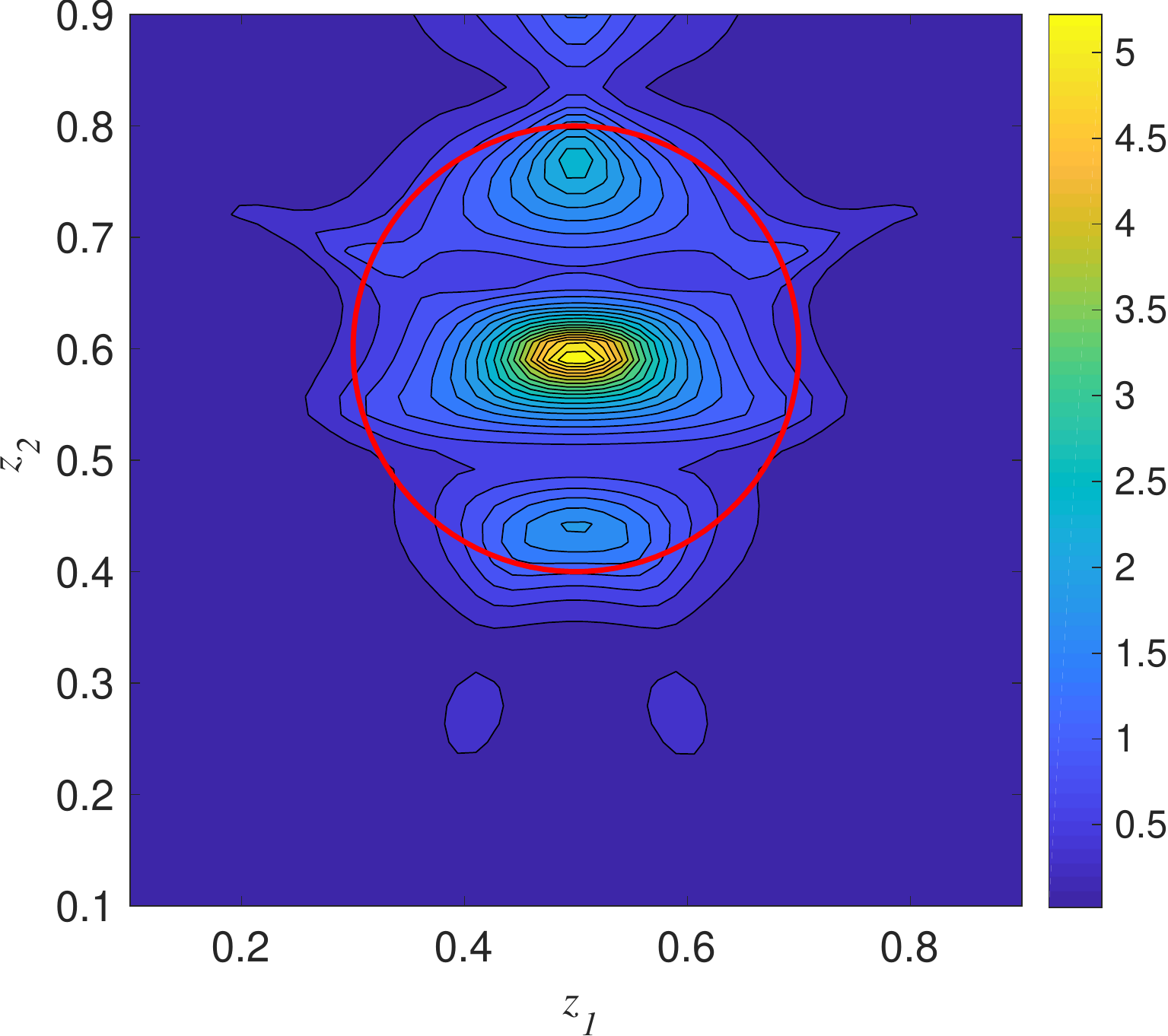}}&
\resizebox{0.3\textwidth}{!}{
\includegraphics{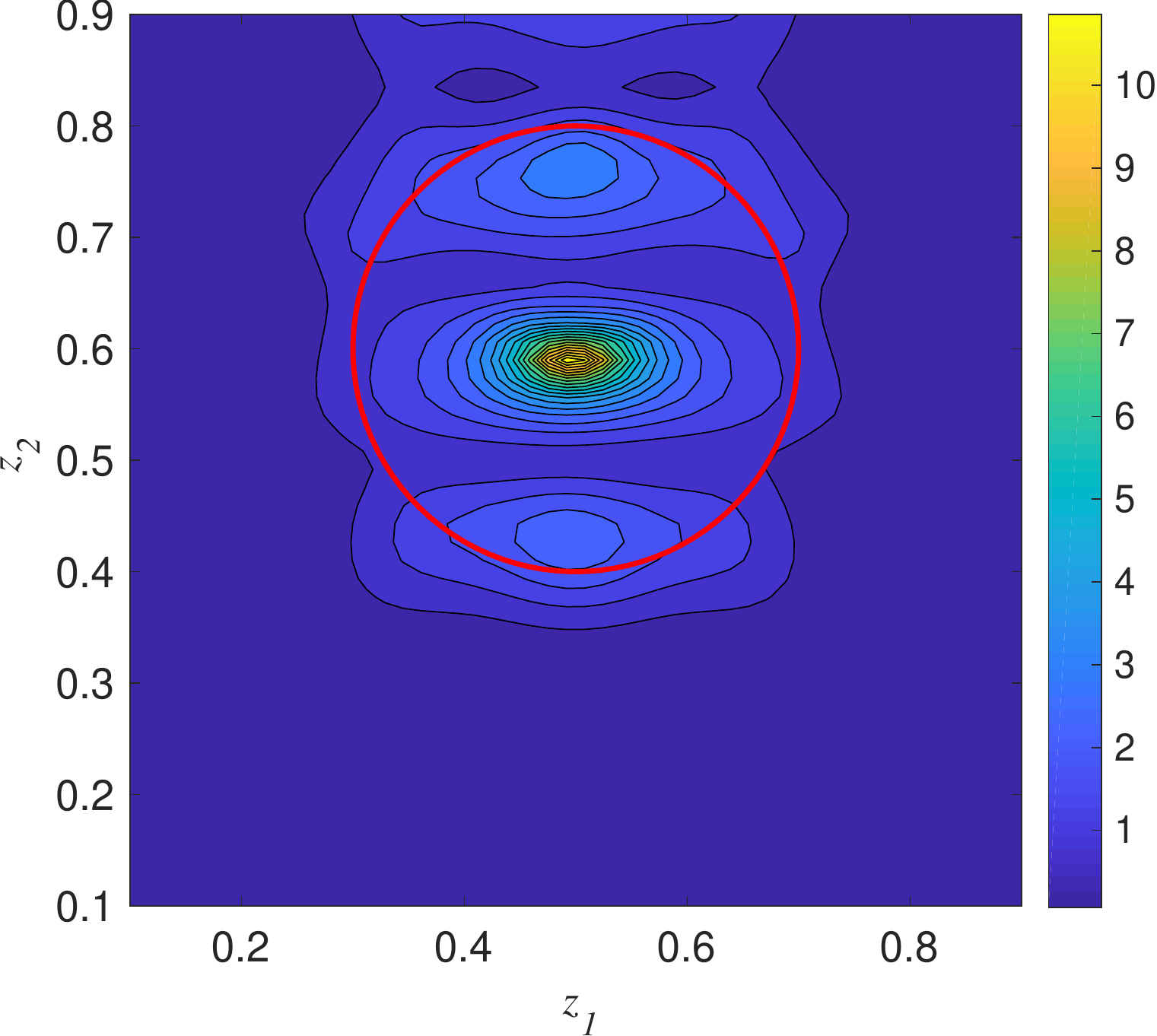}}&
\resizebox{0.3\textwidth}{!}{
\includegraphics{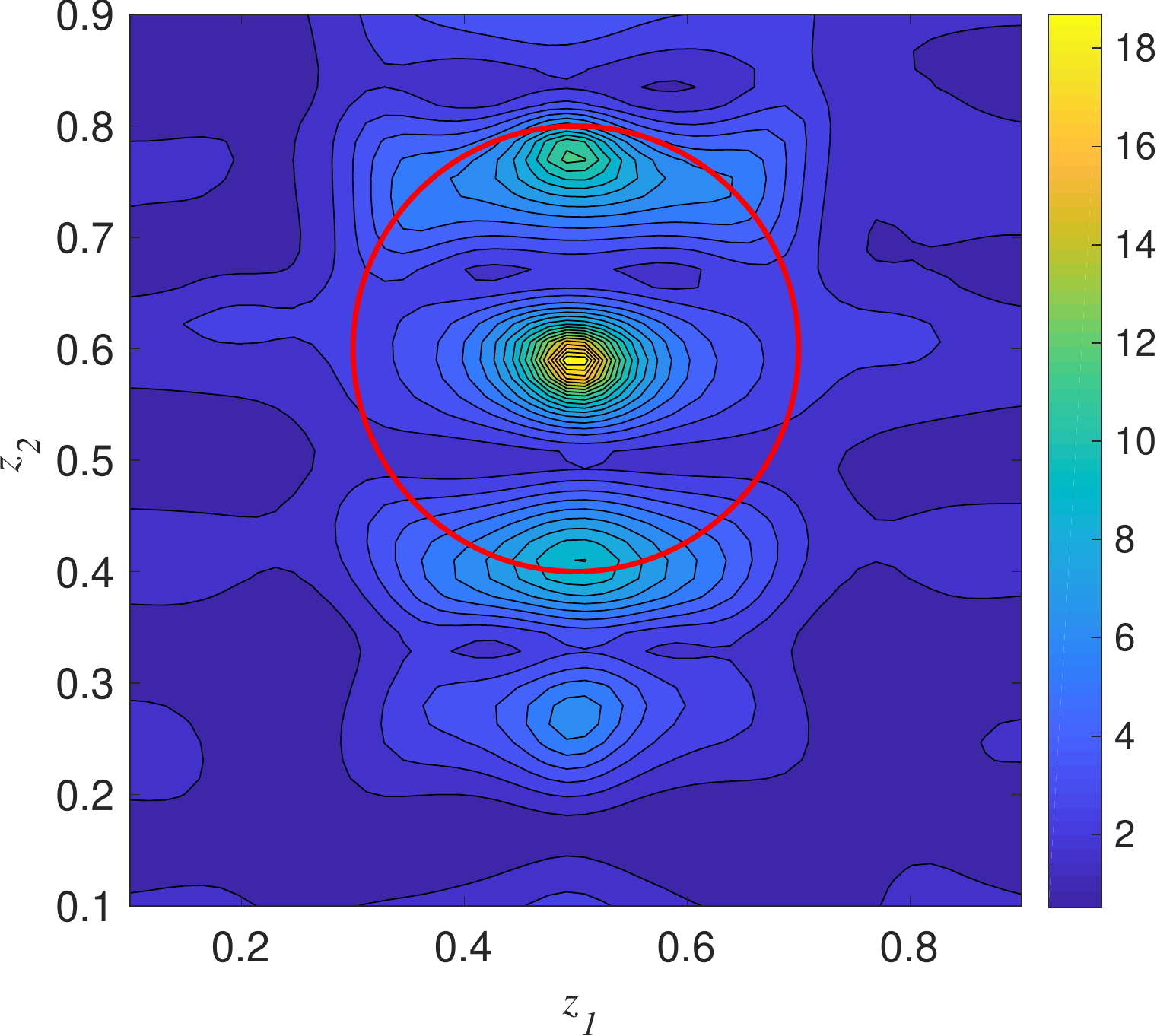}}\\
\resizebox{0.3\textwidth}{!}{
\includegraphics{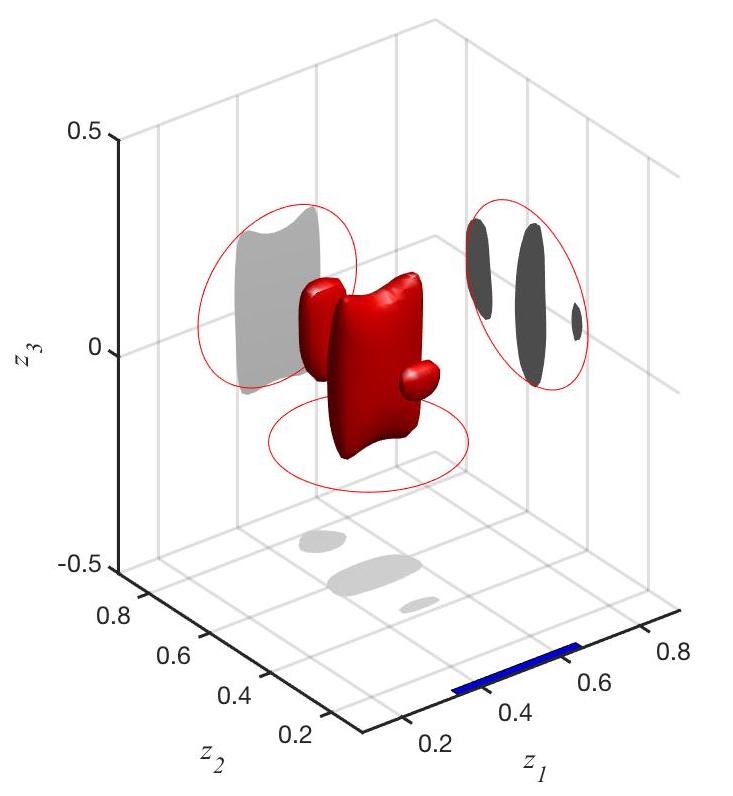}}&
\resizebox{0.3\textwidth}{!}{
\includegraphics{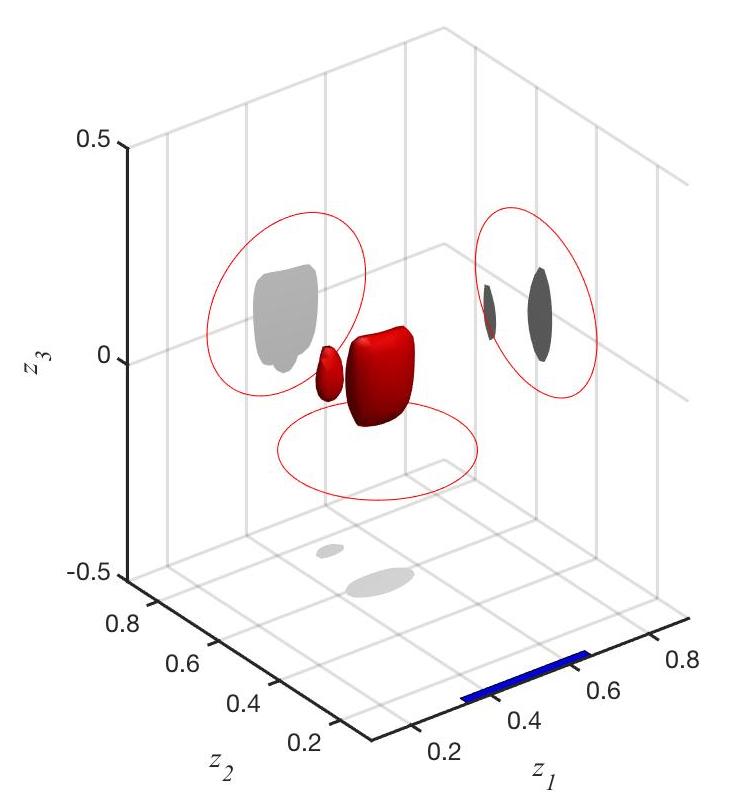}}&
\resizebox{0.3\textwidth}{!}{
\includegraphics{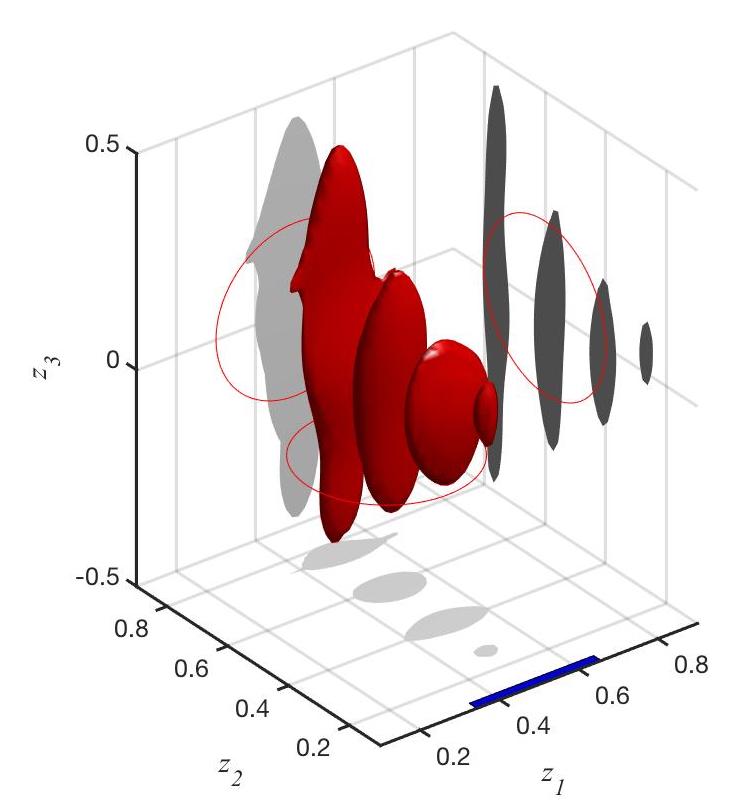}}
\end{tabular}
\end{center}
\caption{Reconstructions of the single sphere example shown in Fig.~\ref{domains} when $k=20$ using 51 singular vectors and the isovalue parameter $C=0.3$.
Top row:  contour plots of the indicator function as a function of $z_1$ and $z_2$ for fixed $z_3=0$.
Bottom Row:  isovalue plots of the indicator function in the search region.  The thick line along the $z_1$ axis shows
the wavelength $2\pi/k$ and the red circle on each coordinate face shows the projection of the scatterer onto the corresponding face.
Left column: no added noise. Center column: with added noise $\eta=0.001$.
Right: with added noise $\eta=0.01$.}
\label{onesphk20}
\end{figure}

\begin{figure}
\begin{center}
\begin{tabular}{cc}
\resizebox{0.45\textwidth}{!}{
\includegraphics{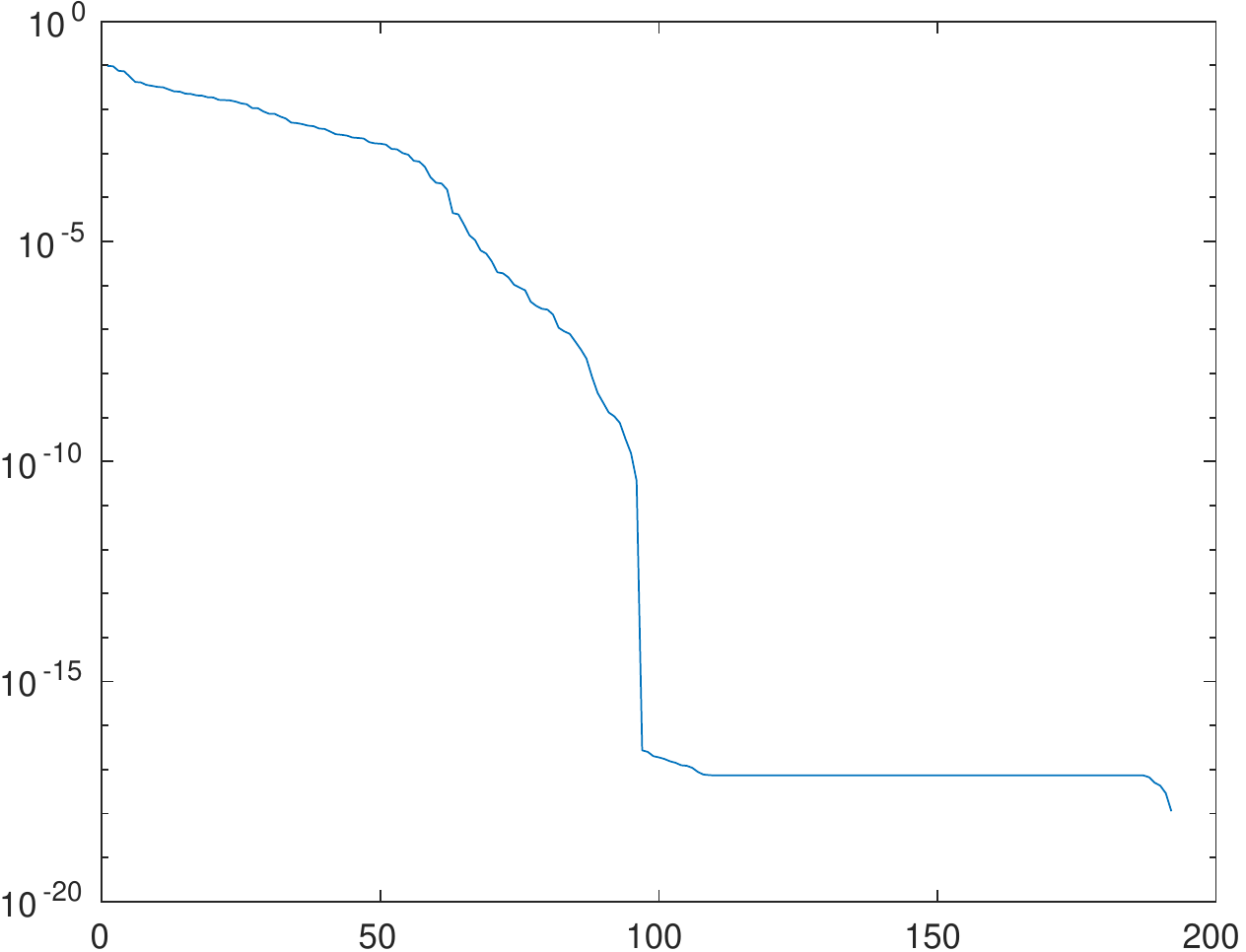}}&
\resizebox{0.45\textwidth}{!}{
\includegraphics{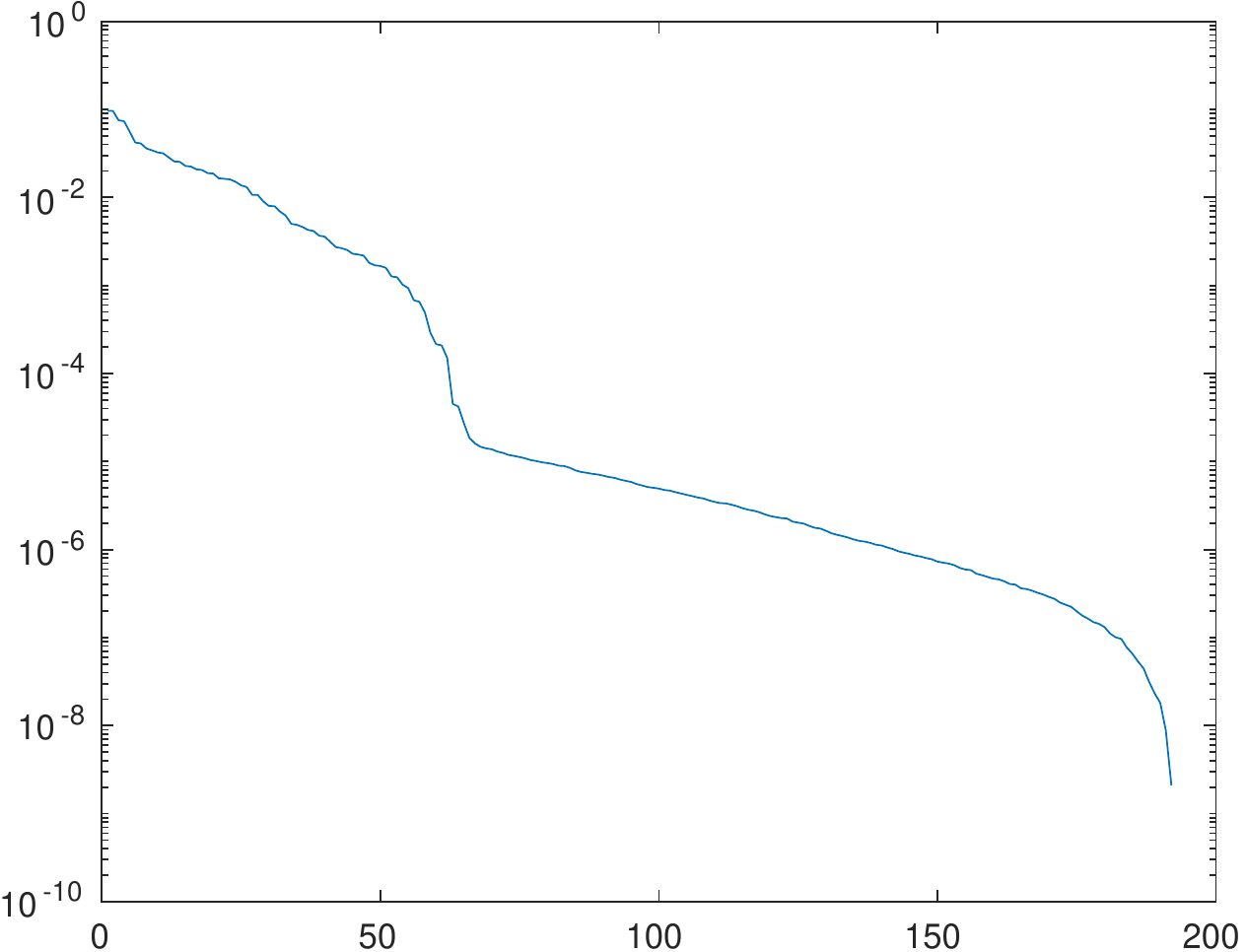}}
\end{tabular}
\end{center}
\caption{Singular values for the single sphere example shown in Fig.~\ref{domains}, left panel, when $k=25$. Left: no added noise. 
Right: with added noise $\eta=0.001$.}
\label{S25}
\end{figure}

\begin{figure}
\begin{center}
\begin{tabular}{ccc}
\resizebox{0.3\textwidth}{!}{
\includegraphics{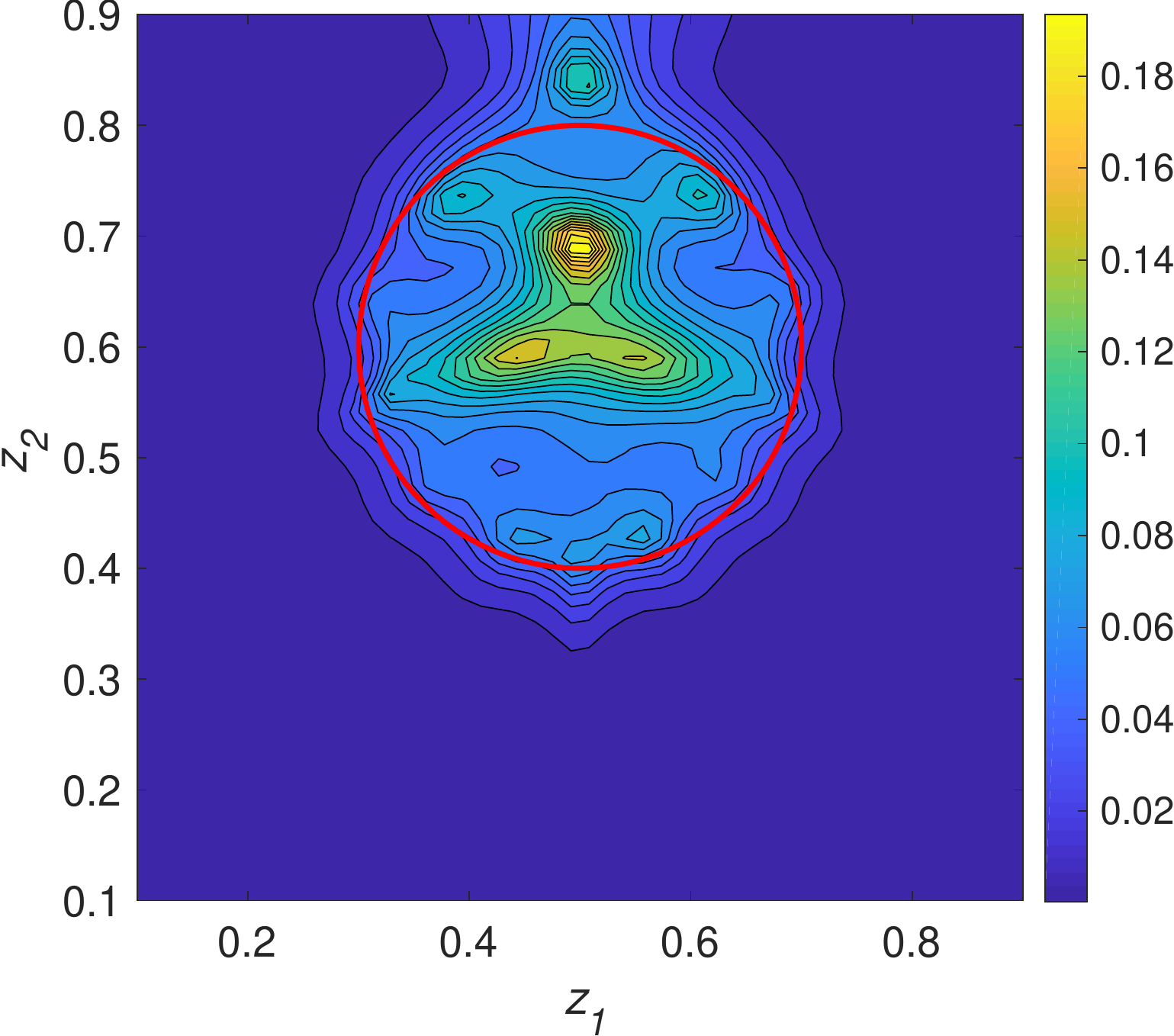}}&
\resizebox{0.3\textwidth}{!}{
\includegraphics{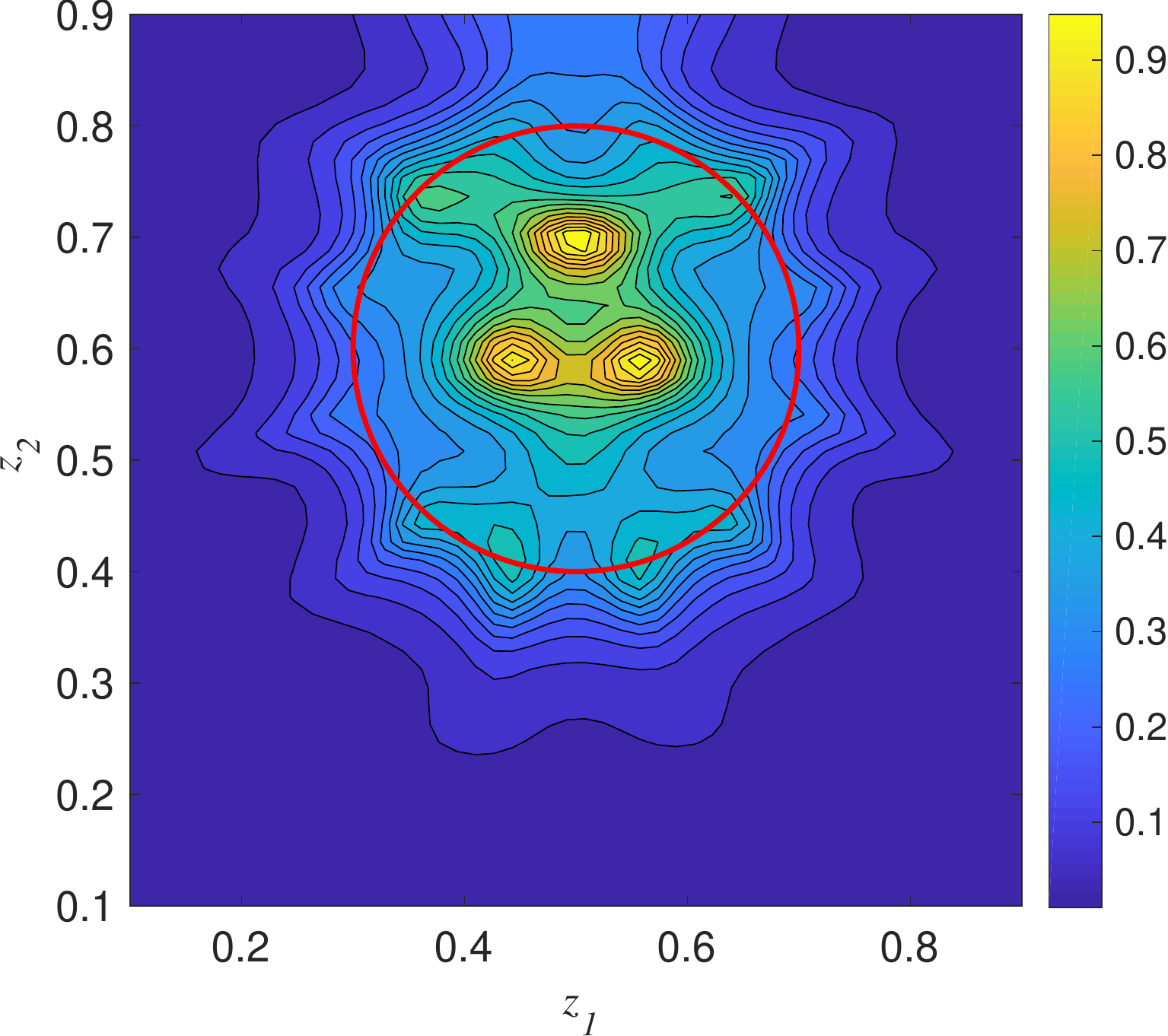}}&
\resizebox{0.3\textwidth}{!}{
\includegraphics{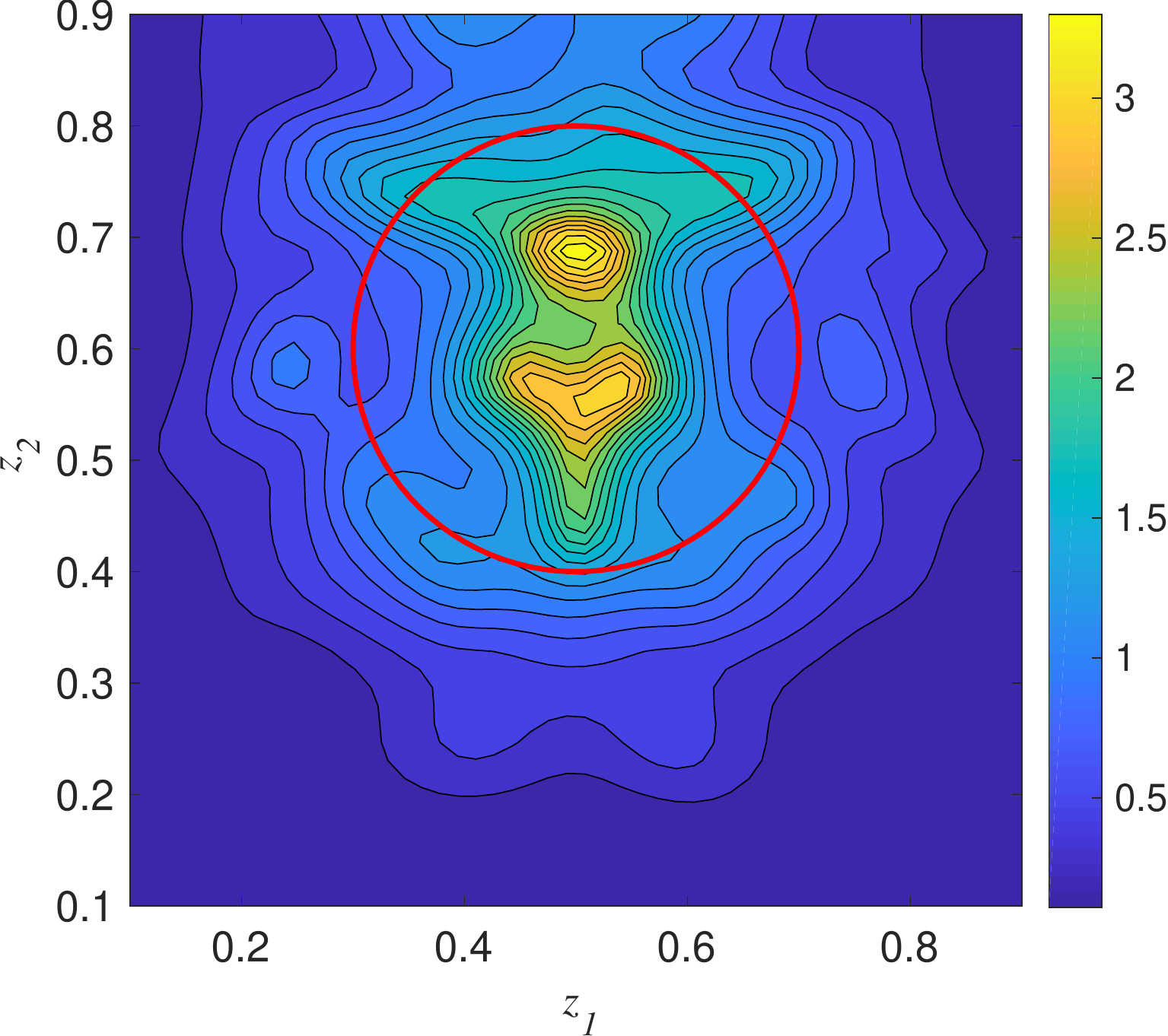}}\\
\resizebox{0.3\textwidth}{!}{
\includegraphics{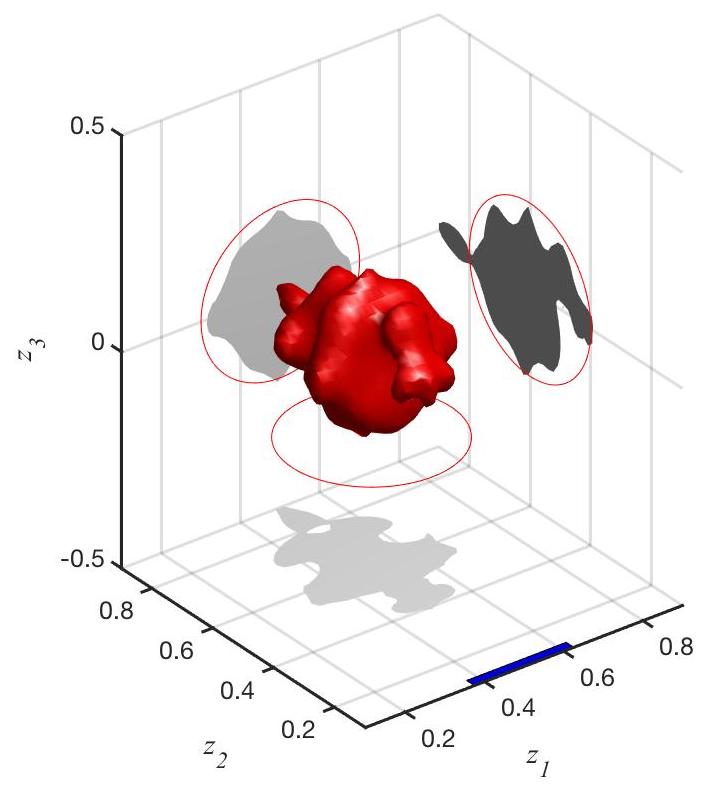}}&
\resizebox{0.3\textwidth}{!}{
\includegraphics{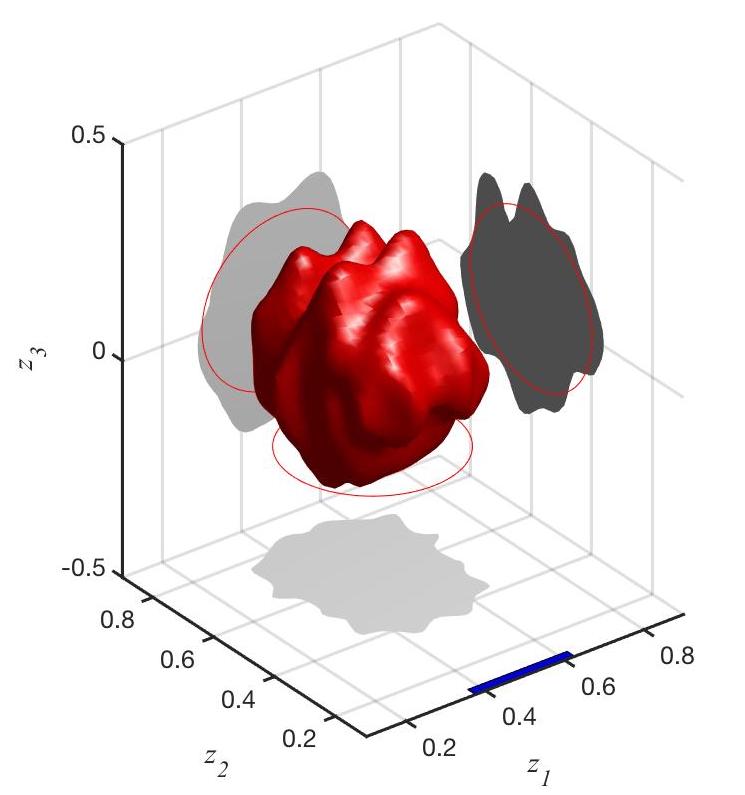}}&
\resizebox{0.3\textwidth}{!}{
\includegraphics{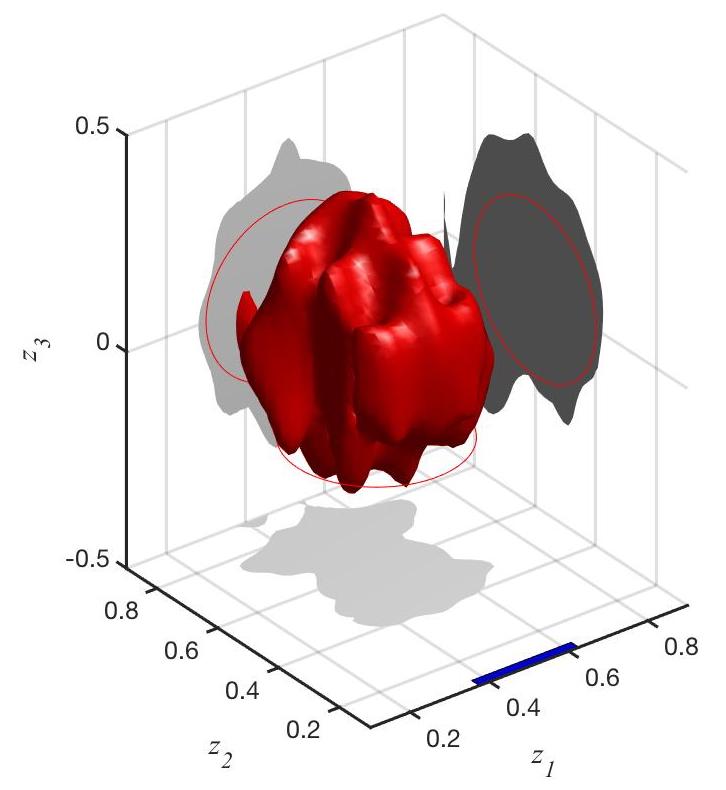}}
\end{tabular}
\end{center}
\caption{Reconstructions of the single sphere example shown in Fig.~\ref{domains} left panel when $k=25$ using
81 singular vectors and the cutoff parameter for the isovalue plots $C=0.3$. For a
description of the features shown in the figures, see the caption of Fig.~\ref{onesphk20}.
Left column: no added noise. Center column: with added noise $\eta=0.001$.
Right column: with added noise $\eta=0.01$.}
\label{onesphk25}
\end{figure}

\subsection{Two spheres}
Next we consider the two spheres example where the exact scatterer is shown in the right-hand panel of Fig.~\ref{domains}.  Perhaps surprisingly, this example can be reconstructed using a lower wavenumber than for the single sphere.  We show results of
reconstructing this scatterer using $k=20$  in Fig.~\ref{twosph}.  
\begin{figure}
\begin{center}
\begin{tabular}{ccc}
\resizebox{0.3\textwidth}{!}{
\includegraphics{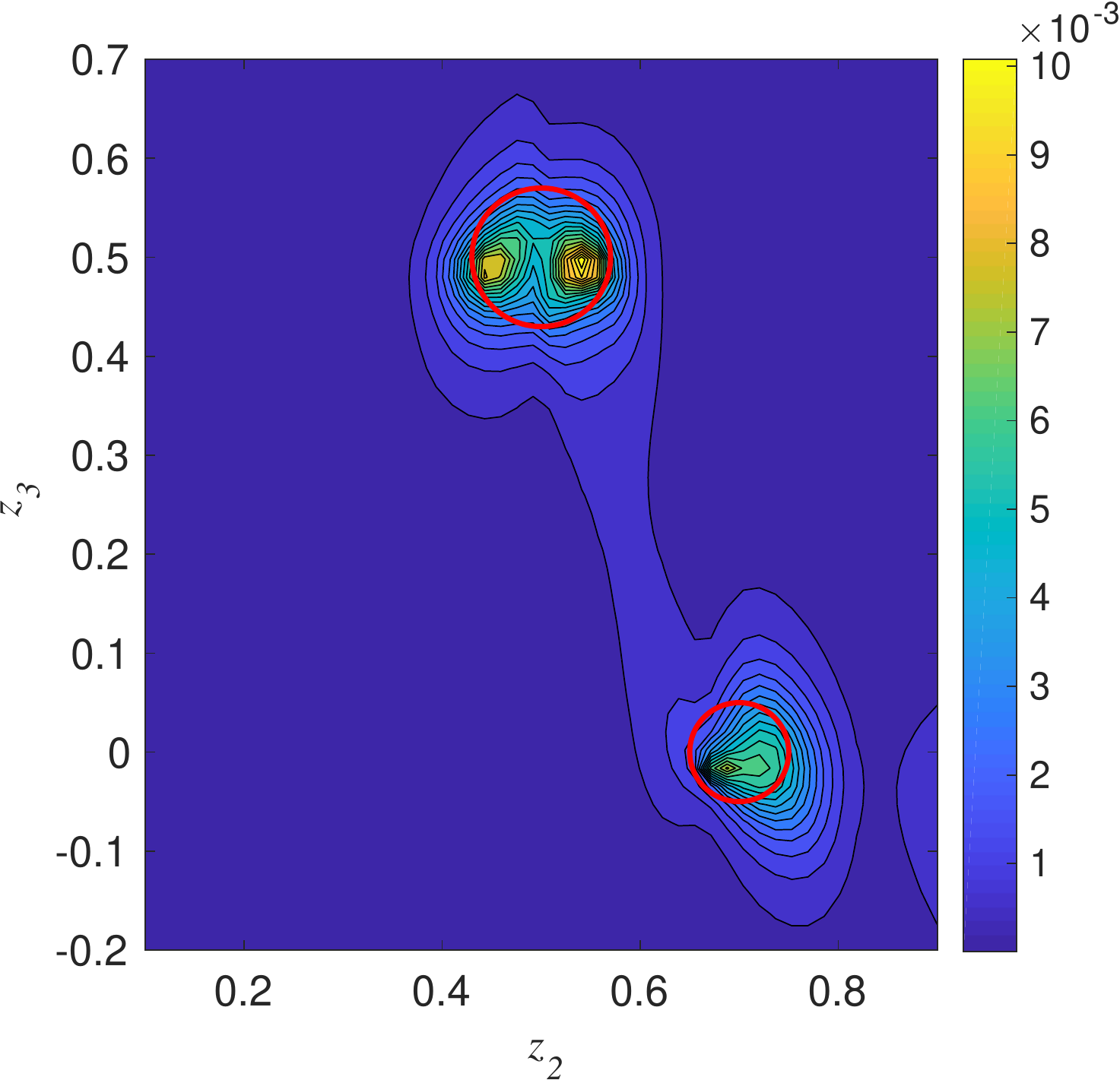}}&
\resizebox{0.3\textwidth}{!}{
\includegraphics{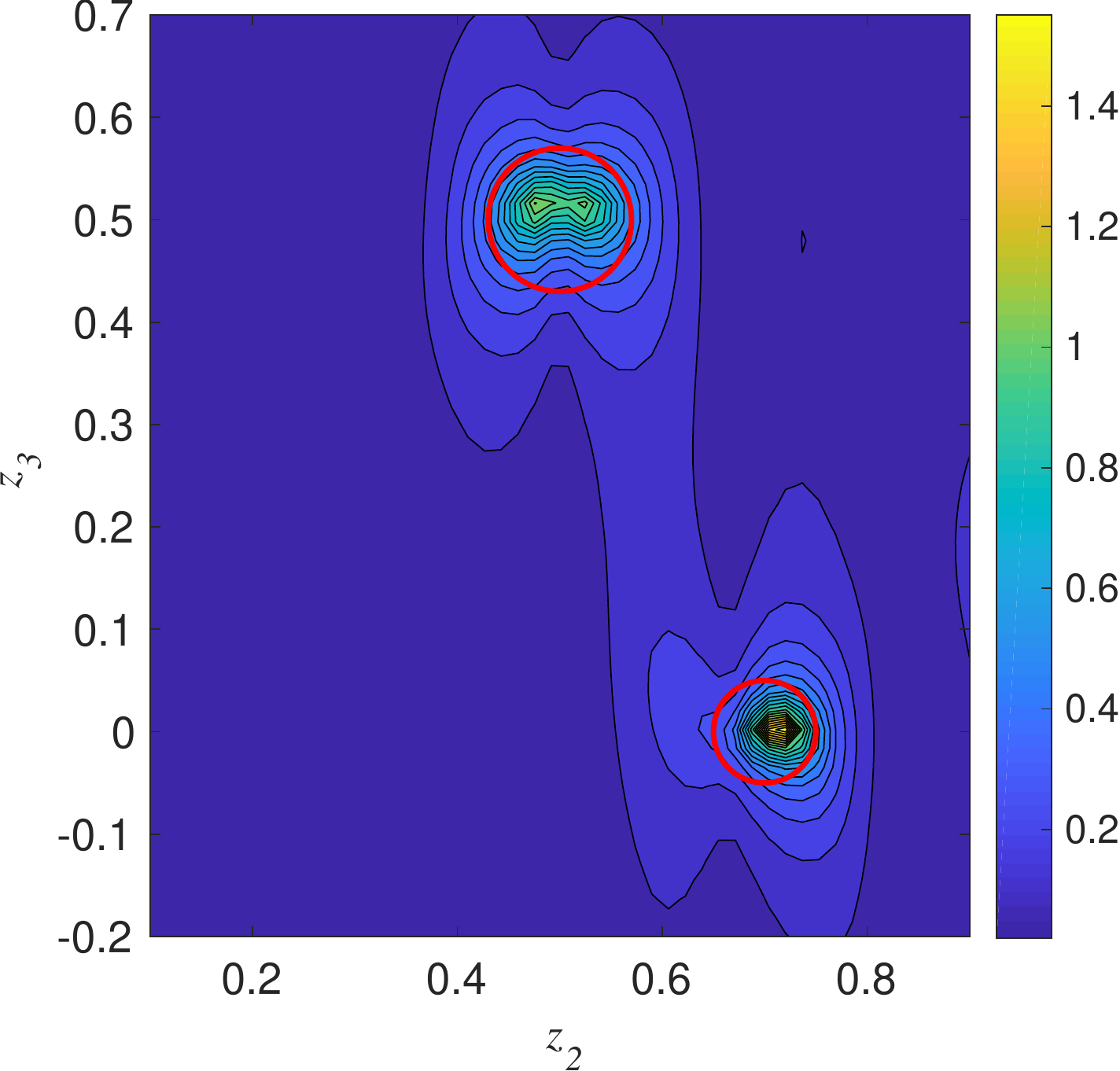}}&
\resizebox{0.3\textwidth}{!}{
\includegraphics{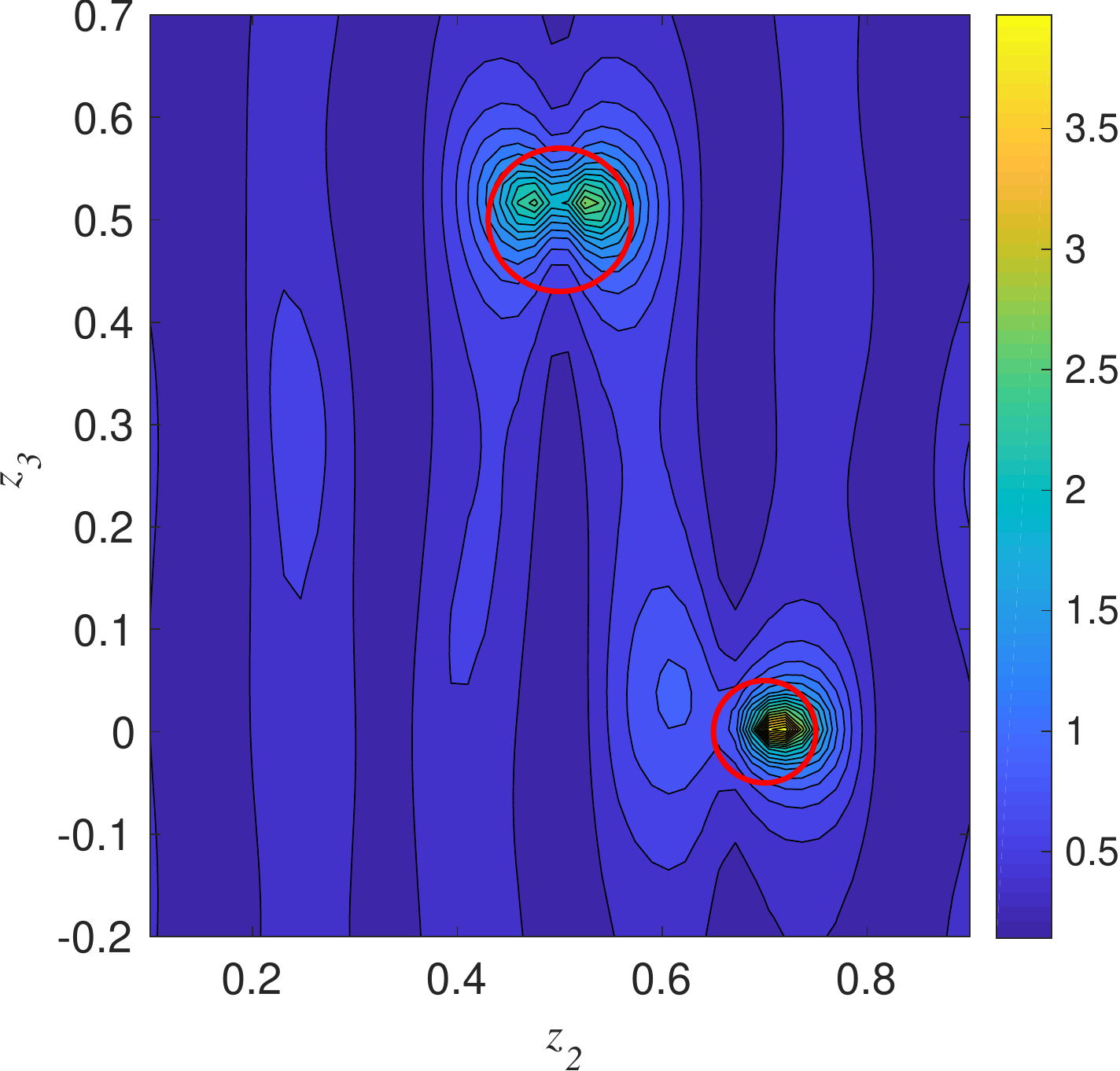}}\\
\resizebox{0.3\textwidth}{!}{
\includegraphics{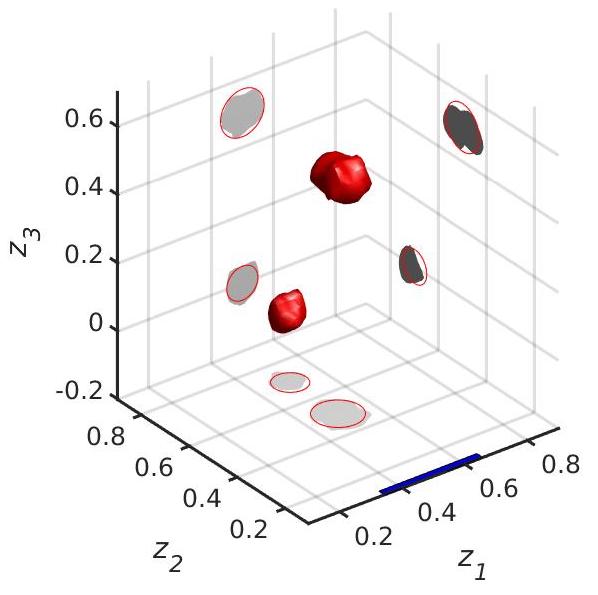}}&
\resizebox{0.3\textwidth}{!}{
\includegraphics{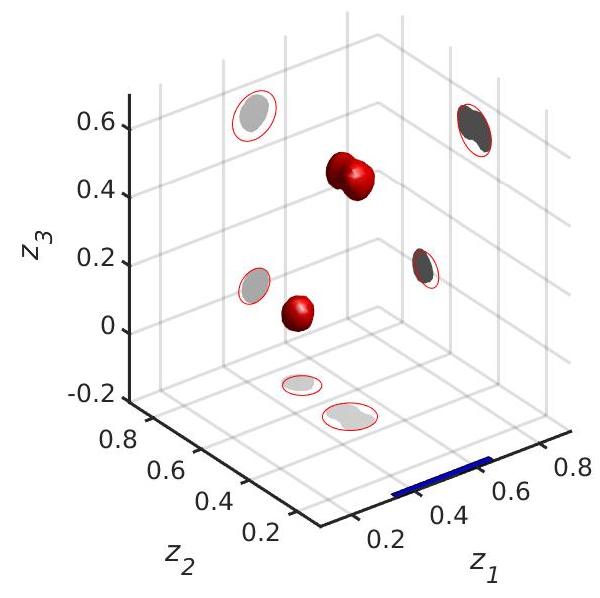}}&
\resizebox{0.3\textwidth}{!}{
\includegraphics{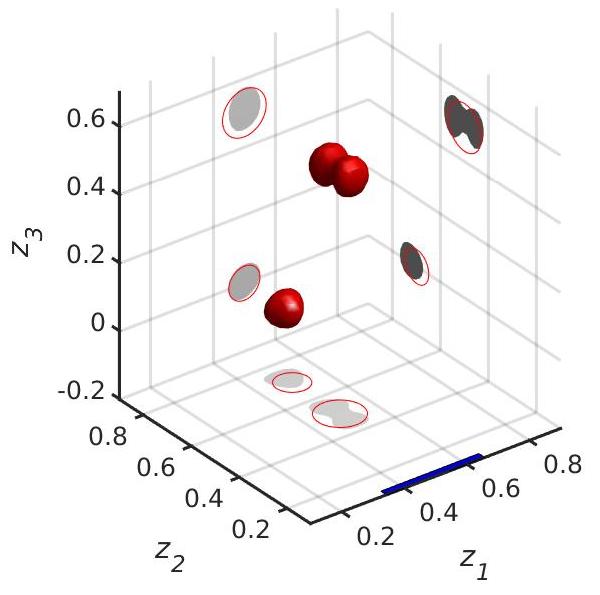}}\\
\end{tabular}
\end{center}
\caption{Reconstructions of the two spheres example shown in Fig.~\ref{domains} right panel when $k=20$ using 56 singular vectors and the isovalue cutoff parameter $C=0.3$. For a
description of the features shown in the figures, see the caption of Fig.~\ref{onesphk20}, {except that the contour
plots are now in the $z_2$-$z_3$ plane.}
Left column: no added noise. 
Center column: with added noise  $\eta=0.001$.
Right column: with added noise $\eta=0.01$.}
\label{twosph}
\end{figure}

\section{Conclusions}\label{Sec-Concl}
 Our analysis and numerical evidence suggests that the LSM can be used to identify the position and size of penetrable obstacles
 in an electromagnetic waveguide.  Clearly the model problem we have examined requires considerable
 elaboration before being useful in applications.  The case 
{when} measurements are made on
 a surface on the opposite side of the obstacle to the receivers could also be investigated (the theory we have {presented} holds in that case as well, but the numerical results in this paper are only for measurement and sources on one side of the obstacle).  However, we suppose that the one sided measurement considered here would be simpler
 in practice.

Although we did not discuss PEC scatterers, exactly the same LSM applies for a PEC or penetrable scatterer.  Theory and numerical results for the PEC case can be found in \cite{FanPhD}.

\section*{Acknowledgments}  The research of P. Monk was partially supported by the Air Force Office of Scientific Research under award number FA9550-17-1-0147,   and that of V. Selgas by project MTM2017-87162-P.

%

\end{document}